\documentclass[final,leqno,onefignum,onetabnum]{siamltex1213}

\usepackage{graphicx}
\usepackage{caption}
\usepackage{subcaption}

\usepackage{titlesec}
\usepackage{cite}
\usepackage{algorithm,algorithmic} 
\usepackage{amsmath,amssymb,amstext} % Lots of math symbols and environments
\usepackage[dvips]{epsfig}
\usepackage{float}
\usepackage{caption}
\usepackage{subcaption}

\usepackage{siunitx}

\def\R{\mathbb{R}}
\def\KK{\mathcal{K}}

\def\PP{\mathcal{P}}
\def\Ss{\mathcal{S}}
\def\Sn{\mathcal{S}^n}
\def\Snp{\mathcal{S}^n_+}
\newcommand{\face}{{\rm face}}
\newcommand{\svec}{{\rm svec\,}}

\newcommand{\GG}{G}
\newcommand{\VV}{V}
\newcommand{\EE}{E}

\DeclareMathOperator{\tr}{{tr}}
\DeclareMathOperator{\Diag}{{Diag}}
\newcommand{\range}{\mathrm{range}\,}
\newcommand{\dist}{\mathrm{dist}}
\newcommand{\proj}{\mathrm{proj}\,}
\newcommand{\noise}{\nu}
\newcommand{\argmin}{\mathop{\rm argmin}}

\newtheorem{thm}{Theorem}[section]

\newtheorem{lem}[thm]{Lemma}
\newtheorem{obs}[thm]{Observation}
\newtheorem{cor}[thm]{Corollary}
\newtheorem{example}[thm]{Example}

\makeindex

\title{Noisy Euclidean distance realization:
	robust facial reduction and the Pareto frontier} 

\author{{D. Drusvyatskiy}\thanks{Department of Mathematics, University of Washington, Seattle, WA 98195-4350, USA. Research  was partially supported by the AFOSR YIP award FA9550-15-1-0237. \url{math.washington.edu/\~ddrusv}
	}
	\and
	{N. Krislock}\thanks{Department of Mathematical Sciences,
		Northern Illinois University, DeKalb, IL 60115, USA.
		\url{www.math.niu.edu/\~krislock}
	}
	\and
	{Y.-L. Voronin}\thanks{Department of Computer Science, University of Colorado, 
		Boulder, CO 80309-0430, USA. \url{cs.colorado.edu/\~yuvo9296}
	}
	\and
	{H. Wolkowicz}\thanks{Department of Combinatorics and Optimization, 
		Waterloo, Ontario N2L 3G1, Canada. Research supported by Natural Sciences Engineering Research
		Council Canada and a grant from AFOSR. \url{orion.math.uwaterloo.ca/\~hwolkowi}}
	}

\begin{document}
\maketitle
\slugger{mms}{xxxx}{xx}{x}{x--x}%slugger should be set to mms, siap, sicomp, sicon, sidma, sima, simax, sinum, siopt, sisc, or sirev

\begin{abstract}
We present two algorithms for large-scale low-rank Euclidean distance matrix completion problems, based on semidefinite optimization. Our first method works by relating cliques in the graph of the known distances to faces of the positive semidefinite cone, yielding a combinatorial procedure that is provably robust and parallelizable. Our second algorithm is a first order method for maximizing the trace---a popular low-rank inducing regularizer---in the formulation of the problem with a constrained misfit. Both of the methods output a point configuration that can serve as a high-quality initialization for local optimization techniques. Numerical experiments on large-scale sensor localization problems illustrate the two approaches.
\end{abstract}

\begin{keywords} Euclidean distance matrices, sensor network localization,
	convex optimization, facial reduction, Frank-Wolfe algorithm, semidefinite programming \end{keywords}

\begin{AMS} 90C22, 90C25 , 52A99  \end{AMS}

\pagestyle{myheadings}
\thispagestyle{plain}
\markboth{NOISY EDM COMPLETION}{D. DRUSVYATSKIY, N. KRISLOCK, Y.-L. VORONIN, AND H. WOLKOWICZ}

\section{Introduction.}
A pervasive task in distance geometry is the inverse problem: given only local pairwise Euclidean distances among a set of points, recover their locations in space. More precisely,
%\begin{changemargin}{0.5cm}{0.5cm}
given a weighted undirected graph $G=(V,E,d)$ on a vertex set $\{1,\ldots,n\}$ and an integer $r$, find (if possible) a set of points $x_1,\ldots,x_n$ in $\R^r$ satisfying \[\|x_i-x_j\|^2=d_{ij}, \quad\textrm{ for all edges } ij\in E,\]
%\end{changemargin}
where $\|\cdot\|$ denotes the usual Euclidean norm on $\R^r$.
In most applications, the given squared distances $d_{ij}$ are inexact, and one then seeks points $x_1,\ldots, x_n$ satisfying the distance constraints only approximately.
%Measuring distance violations in the $l_1$ and $l_{\infty}$ norms, instead of the $l_2$ norm above, is popular as well. We focus on the $l_2$ case for concreteness.
This problem appears under numerous names in the literature, such as Euclidean Distance Matrix (EDM) completion and graph realization \cite{AlWo:99,MR99c:05135,dattorro:05}, and is broadly applicable for example in  wireless networks, statistics, robotics, protein reconstruction, and dimensionality reduction in data analysis; the recent survey \cite{surv} has an extensive list of relevant references. 
Fixing notation, we will refer to this problem as {\em EDM completion}, throughout. 

The EDM completion problem can be modeled
as the nonconvex feasibility problem: find a symmetric $n\times n$ matrix $X$ satisfying
\begin{equation}\label{eqn:SDP_form}
\left\{ \begin{array}{l}
X_{ii} + X_{jj} - 2X_{ij} = d_{ij},\qquad \textrm{ for all } ij\in E,\\
Xe = 0,\\
\rank X \leq r, \\
X \succeq 0,
\end{array} \right\} 
\end{equation}
where $e$ stands for the vector of all ones. Indeed, if $X=PP^T$ is a maximal rank factorization of such a matrix $X$, then the rows of $P$ yield a solution to the EDM completion problem. The constraint $Xe=0$ simply ensures that the rows of $P$ are centered around the origin.
Naturally a convex relaxation is obtained by simply ignoring the rank constraint. The resulting problem is convex (a semidefinite program (SDP) in fact) and so more tractable. %Moreover, one can ensure that the feasible region is compact by adding the constraint $Xe=0$  
For many instances, particularly coming from dense wireless networks, this relaxation is exact, that is the solution of the convex rank-relaxed problem automatically has the desired rank $r$ \cite{sy07}. %It is intuitively clear that the relaxation is more likely to be exact for dense graphs $G$.
%The number of variables in the SDP is typically very large, say on the order of $10^{10}$ corresponding to $10^5$ vertices in the graph. 
Consequently, semidefinite programming techniques have proven to be extremely useful for this problem; see for example
\cite{BYACM:06,BYIEEE:06,BYSJC:08,BY:04,by03,
	KH:2010,pt09,sy07,wzyb08}. For large networks, however, the SDPs involved can
become intractable for off-the-shelf methods. Moreover, this difficulty is compounded by the inherent {\em ill-conditioning} in the SDP relaxation of \eqref{eqn:SDP_form}---a key theme of the paper. %To illustrate, consider a {\em clique} in the graph $G$, that is a subset $\chi$ of vertices so that every two vertices in $\chi$ are connected by an edge. Since all the pairwise distances between points in $\chi$ are known, at least intuitively, the clique is a ``degenerate'' part of the problem. 
For example, it is easy to see that each clique in $G$ on more than $r+2$ vertices certifies that the SDP is {\em not} strictly feasible, provided the true points of the clique were in general position in $\R^r$. 

In the current work, we attempt to close the computational
gap by proposing a combinatorial algorithm and an efficient
first-order method for the EDM completion problem.
The starting point is the observation that the cliques in $G$ play a special role in the completion problem. Indeed, from each sufficiently large clique in the graph $G$, one can determine a face of the positive semidefinite cone containing the entire feasible region of \eqref{eqn:SDP_form}. This observation immediately motivated the algorithm of \cite{KH:2010}. The procedure proceeds by collecting a large number of cliques in the graph and intersecting the corresponding faces two at a time (while possibly growing cliques), each time causing a dimensional decrease in the problem. If the SDP relaxation is exact and the graph is sufficiently dense, the method often terminates with a unique solution without having to invoke an SDP solver.
An important caveat of this geometric approach
is that near-exactness of the distance
measurements is essential for the algorithm to work,
both in theory and in practice, for the simple reason that randomly perturbed
faces of the positive semidefinite cone typically intersect
only at the origin.
Remarkably, using dual certificates,
we are able to design a method complementary to
\cite{KH:2010} for the problem \eqref{eqn:SDP_form}
 that under reasonable conditions, is provably
robust to noise in the distance measurements, in the sense that the output error
is linearly proportional to the noise level. Moreover, in contrast to the algorithm \cite{KH:2010},
the new method is conceptually easy to parallelize.
In the late stages of writing the current paper, we became aware of the related work \cite{face_sing}. There the author proposes a robust algorithm for the EDM completion problem that is in the same spirit as ours, but is stated in the language of rigidity theory. As a byproduct, our current work yields an interpretation of the algorithm \cite{face_sing} in terms of facial reduction iterations and SDP techniques. Moreover, we offer additional improvements via a nonrigid clique union subroutine (Subsection~\ref{subsubsec:clique_union}), which we found essential for the success of the algorithm.

In the second part of the paper, we propose a first order method for solving the noisy EDM completion problem. %To this end, consider the problem \cite{AEGMWYAB:06,WeinbergerEtAl}: 
To this end, we consider maximizing the trace---a popular low-rank
inducing regularizer \cite{AEGMWYAB:06,WeinbergerEtAl}---in
the formulation of the problem:
\begin{align}
\text{maximize}~~~ &\tr X\notag\\
\text{subject to}~~&\sum_{ij\in E} |X_{ii}+X_{jj}-2X_{ij}-d_{ij}|^2\leq \sigma \label{eqn:robust_const}\\
&Xe=0\notag\\
&X\succeq \notag 0.
\end{align} 
Here $\sigma$ is an a priori chosen tolerance reflecting the total noise level. 
Notice, that this formulation directly contrasts the usual min-trace regularizer in compressed sensing; nonetheless it is very natural. An easy computation shows that in terms of the factorization $X=PP^T$, the equality $\tr(X)= \frac{1}{2n}\sum^n_{i,j=1} \|p_i-p_j\|^2$ holds, where $p_i$ are the rows of $P$. Thus trace maximization serves to ``flatten'' the realization of the graph. 
%Of course, instead one may instead consider the formulation:
%\begin{align*}
%\min~~ &\sum_{ij\in E} |X_{ii}-2X_{ij}+X_{jj}- d_{ij}|^2-\lambda \tr(X)\\
%\textrm{s.t. }~&Xe=0\\
%&X\succeq 0,
%\end{align*} 
We note in passing that we advocate using \eqref{eqn:robust_const} instead of perhaps the more usual regularized problem 
\begin{align}
\text{minimize}~~~ &\sum_{ij\in E} |X_{ii}+X_{jj}-2X_{ij}-d_{ij}|^2-\lambda\tr X\notag\\
\text{subject to}~~&Xe=0,\quad X\succeq \notag 0\notag.
\end{align} 
The reason is that choosing a reasonable value of the trade-off parameter $\lambda$ can be difficult, whereas an estimate of $\sigma$ is typically available from a priori known information on the noise level.

As was observed above, for $\sigma=0$ the problem formulation \eqref{eqn:robust_const} notoriously fails strict feasibility. In particular, for small $\sigma\geq 0$ the feasible region is very thin and the solution to the problem is unstable. As a result, iterative methods that maintain feasibility are likely to exhibit serious difficulties. Keeping this in mind, we propose an {\em infeasible} first-order method, which is not directly effected by the poor conditioning of the underlying problem.

To this end, consider the following parametric problem, obtained by ``flipping'' the objective and the quadratic constraint in \eqref{eqn:robust_const}:
\begin{align*}
v(\tau):=\quad \text{minimize} ~~~~& \displaystyle\sum_{ij\in E}  \,|X_{ii}+X_{jj}-2X_{ij}-d_{ij}|^2 \\\notag
\text{subject to}~~~~
& \tr X=\tau\notag\\ 
& Xe=0\\ \notag
& X\succeq 0. \notag
\end{align*}
Notice that the problem of evaluating $v(\tau)$ is readily amenable to first order methods, in direct contrast to \eqref{eqn:robust_const}. Indeed, the feasible region is geometrically simple. In particular, linear optimization over the region only requires computing a maximal eigenvalue. Hence the evaluation of $v(\tau)$
is well adapted for the {\em Frank-Wolfe method}, a projection-free first order algorithm.
%It is easy to see that solving \eqref{eqn:robust_const} then amounts to finding the largest value of $\tau$ satisfying $v(\tau)\leq \sigma$. 
%This formulation is perfectly adapted for the {\em Frank-Wolfe method}.
Indeed, the gradient of the objective function is very sparse (as sparse as the edge set $E$) and therefore optimizing the induced linear functional  over the feasible region then becomes a cheap operation. %For details see the preprint \cite{noisy}.
Now, solving \eqref{eqn:robust_const} amounts to finding the largest value of $\tau$ satisfying $v(\tau)\leq \sigma$, a problem that can be solved by an approximate Newton method.
Analogous root finding strategies can be found, for example, in
\cite{var_prop,par,lsqr_con,spgl1:2007}. Using this algorithm, we investigate the apparent superiority of the max-trace regularizer over the min-trace regularizer with respect to both low-rank recovery and efficient computation.

%We will illustrate our methods on sensor localization problems. 
%Consider a wireless network of sensors taking measurements (e.g. temperature, pressure, etc). Such networks are typically extremely large and dense. The sensor data is only relevant when the 
%sensor's location in space is known, and hence knowledge of sensor positions becomes imperative. Since GPS tracking is usually unavailable or too expansive to be used, one may try to recover the sensor positions solely from local pairwise distances between neighboring sensors (i.e.  sensors within a certain distance of each other, called the ``radio range''). In most situations, such distance measurements have some degree of error, for example due to power and memory constraints, or perhaps slow drifting of the system.  
%Recovering the location of sensors is the {\em network localization problem} --- an instance of a large-scale EDM completion problem with noisy data and side constraints corresponding to a few sensors (called anchors) whose location is a priori known. The target embedding dimension $r$ in this problem is typically two or three.

The outline of the paper is as follows. Section~\ref{sect:problem_statement} collects some preliminaries on the facial structure of the positive semidefinite cone and the SDP relaxation of the EDM completion problem. Section 
\ref{sect:rfr} presents the proposed robust facial reduction algorithm and provides
some numerical illustrations.
Section \ref{sect:Pareto} describes the proposed Pareto search technique with Frank-Wolfe iterations, and presents numerical experiments.

\section{Preliminaries.}\label{sect:problem_statement}

In this section, we record some preliminaries and formally state the 
EDM completion problem. 

\subsection{Geometry of the positive semidefinite cone.}
The main tool we use in the current work 
(even if indirectly) is semidefinite programming (SDP).
To this end, let $\mathcal{S}^n$ denote the Euclidean 
space of $n\times n$ real symmetric matrices endowed 
with the trace inner product 
$\langle A,B\rangle=\tr AB$ and the Frobenius norm 
$\|A\|_F=\sqrt{\tr A^2}$. 
The convex cone of $n\times n$ positive semidefinite (PSD)
matrices will be denoted by $\mathcal{S}^n_+$. This cone
defines a partial ordering:
for any $A,B\in\mathcal{S}^n$ the binary relation
$A\succeq B$ means $A-B\in\mathcal{S}^n_+$.
%We let $e\in\R^n$ denote the vector of all ones. 
A convex subset $\mathcal{F}$ of $\mathcal{S}^n_+$ is a {\em face} of $\mathcal{S}^n_+$ if $\mathcal{F}$ contains any line segment in $\mathcal{S}^n_+$ whose relative interior intersects $\mathcal{F}$, and a face $\mathcal{F}$ of $\mathcal{S}^n_+$ is \emph{proper}
if it is neither empty nor all of $\mathcal{S}^n_+$.
All faces of $\mathcal{S}^n_+$ have the (primal) form 
\begin{equation} \label{eqn:prim_form1}
\mathcal{F}=\left\{ 
U\begin{bmatrix} A & 0 \\ 0 & 0 \end{bmatrix}U^T: 
A\in\mathcal{S}^k_+
\right\},
\index{trace inner product}
\index{Frobenius norm, $\|\cdot\|_F$}
\index{$\|\cdot\|_F$, Frobenius norm}
\index{$n\times n$ real symmetric matrices, $\Ss^n$}
\index{$\Ss^n$, real symmetric matrices}
\index{positive semidefinite matrices, $\Ss^n_+$}
\index{$\Ss^n_+$, positive semidefinite matrices}
\index{L\"owner cone ordering, $\succeq$}
\index{$\succeq$, L\"owner cone ordering}
\index{vector of all ones, $e$}
\index{$e$, vector of all ones}
\end{equation}
for some $n\times n$ orthogonal matrix $U$ and some 
integer $k\in\{0,1,\ldots,n\}$. Any face $\mathcal{F}$ of $\mathcal{S}^n_+$ can also be written in dual form as $Y^{\perp}\cap \mathcal{S}^n_+$ for some PSD matrix $Y\in\mathcal{S}^n_+$. Indeed, suppose that $\mathcal{F}$ has the representation \eqref{eqn:prim_form1}. 
%subdivide $U$ into two parts 
%$U=\begin{bmatrix} U_R & U_N  \end{bmatrix}$, 
%where $U_R$ has $k$ columns and $U_N$ has $n-k$ columns.
Then we may equivalently write $\mathcal{F}=Y^{\perp}\cap \mathcal{S}^n_+$, with $Y:=U\begin{bmatrix} 0 & 0 \\ 0 & B \end{bmatrix}U^T$ for any nonsingular matrix $B$ in $\mathcal{S}^{n-k}_+$.
%Going back and forth between such primal and dual
%representation of faces is easy, and 
%or more cheaply a $QR$ factorization, for example. 
In general, if a face has the form 
$\mathcal{F}=Y^{\perp}\cap \mathcal{S}^n_+$ 
for some PSD matrix $Y$, 
then we say that $Y$ {\em exposes} $\mathcal{F}$.
Finally, for any convex subset 
$\Omega\subset \Ss^{n}_{+}$, 
the symbol $\face(\Omega;\Ss^{n}_{+})$ will denote the 
minimal face of $\Ss^{n}_{+}$
containing $\Omega$. The cone $\face(\Omega;\Ss^{n}_{+})$ then coincides with $\face(X;\Ss^{n}_{+})$, where $X$ is any maximal rank matrix in $\Omega$.

\subsection{EDM completion problem.}
Throughout, we fix an integer $r\geq 0$ and a weighted undirected
graph $G=(V,E,d)$ on a node set $V=\{1,\ldots,n\}$,
with an edge set $E\subseteq \{ij: 1\leq i <j \leq n\}$
and a vector $d\in\R^\EE$ of nonnegative weights.
The vertices represent points in an $r$-dimensional
space ${\R}^r$,
while the presence of an edge $ij$ joining the vertices
$i$ and $j$ signifies that the physical distance between
the points $i$ and $j$ is available.

The {\em EDM completion problem} is to find
a set of points in 
$x_1,\ldots,x_n\in\R^r$ satisfying 
\begin{equation*} %\label{eqn:noiselessSNL}
\|x_i-x_j\|^2=d_{ij},\quad\quad\textrm{ for all } ij\in E.
\end{equation*}
\index{sensor network localization, SNL! noiseless}   
Such a collection of points $x_1,\ldots,x_n$ is 
said to {\em realize} the graph $G$ in $\R^r$. Notice that without loss of generality, such realizing points $x_1,\ldots,x_n$ can always be translated so that they are centered around the origin, meaning $\sum_i x_i=0$.
%Usually there is in addition a distinguished subset of 
%sensors, called anchors, 
%whose location in space is a priori known. 
%For ease of exposition, we will treat such anchors as 
%regular sensors, throughout. 
%This is fairly innocuous, since if the solution to the 
%anchorless problem is unique up to orthogonal 
%transformations, which is the usual setting, 
%then we may first localize the anchorless network, 
%and then easily translate and rotate the configuration 
%of sensors to align the anchors with their known %positions. 
%We comment further on incorporating anchors into our 
%framework in the conclusion. 

The EDM completion problem is equivalent to finding 
a matrix $X\in \mathcal{S}^n$ satisfying the system:
\begin{equation} \label{eqn:main_eq}
\left\{ \begin{array}{l}
X_{ii}+X_{jj}-2X_{ij}=d_{ij},\quad \textrm{ for all } ij\in E,\\
Xe=0,\\
\rank X\leq r, \\
X\succeq 0.
\end{array} \right\} 
\end{equation}
Here $e\in\R^n$ denotes the vector of all ones. 
Indeed, suppose that $X$ satisfies this system. 
Then since $X$ is positive semidefinite and 
has rank at most $r$, 
we may form a factorization $X=PP^T$ for some $n\times r$ 
matrix $P$. 
It is easy to verify that the rows of $P$ realize $G$ in $\R^r$. 
Conversely, if some points $x_1,\ldots,x_n\in\R^r$ 
realize $G$ in $\R^r$, then we may center them around the 
origin and assemble them into the matrix 
$P=[x_1;\ldots;x_n]^T\in\R^{n\times r}$. 
The resulting {\em Gram matrix} 
$X:=PP^T$ is feasible for the above system. 
For more details, see for example \cite{KH:2010}.

The EDM completion problem is nonconvex 
and is NP-hard in general 
\cite{MR84b:94003,yemeni:79}. 
A convex relaxation is obtained simply by ignoring the 
rank constraint yielding a convex SDP feasibility problem:  
\begin{equation}\label{eqn:SDP_conv}
\left\{ \begin{array}{l}
X_{ii}+X_{jj}-2X_{ij}=d_{ij},\quad \textrm{ for all } ij\in E,\\
Xe=0,\\
X\succeq 0.
\end{array} \right\} 
\end{equation}
For many EDM completion problems on fairly dense graphs, 
this convex relaxation is ``exact'' \cite{SoYe:05}. 
For example the following is immediate. 
\smallskip

\begin{obs}[Exactness of the relaxation]\label{obs:exa} %{\hfill \\ }
	If the EDM completion problem~\eqref{eqn:main_eq}
	is feasible, then the following are equivalent:
	\begin{enumerate}
		\item 
		No realization of $G$ in $\R^l$, for $l >r$, 
		spans the ambient space $\R^l$.
		\item 
		Any solution of the relaxation \eqref{eqn:SDP_conv} 
		has rank at most $r$ and consequently any solution of \eqref{eqn:SDP_conv} yields a 
		realization of $G$ in $\R^r$.
	\end{enumerate}
\end{obs}

\smallskip
In theory, the exactness of the relaxation is a great 
virtue. 
From a computational perspective, however, 
exactness implies that the SDP formulation 
\eqref{eqn:SDP_conv} does not admit a positive 
definite solution, 
i.e., that strict feasibility fails.
Moreover, it is interesting to note that a very minor 
addition to the assumptions of Observation~\ref{obs:exa} 
implies that the SDP~\eqref{eqn:SDP_conv} admits a unique solution \cite{SoYe:05}.
We provide a quick proof for completeness, 
though the reader can safely skip it.

\smallskip
\begin{obs}[Uniqueness of the solution]\label{obs:unique_sol} %{\hfill \\}
	If the EDM completion problem~\eqref{eqn:main_eq}
	is feasible, then the following are equivalent:
	\begin{enumerate}
		\item\label{it:1} 
		The graph $G$ cannot be realized in 
		$\R^{r-1}$, and moreover for any $l >r$ 
		no realization in $\R^l$ spans the ambient space 
		$\R^l$.
		\item\label{it:2} 
		The relaxation \eqref{eqn:SDP_conv} has a unique solution.
	\end{enumerate}
%	Moreover, if either of the above conditions holds, 
%	then the EDM completion problem has a unique solution, 
%	up to a linear isometry.
\end{obs}
\begin{proof}
	The implication $\ref{it:2}\Rightarrow \ref{it:1}$ is 
	immediate. To see the converse implication 
	$\ref{it:1}\Rightarrow \ref{it:2}$, suppose that the 
	SDP \eqref{eqn:SDP_conv} admits two solutions $X$ and $Y$. Define $\mathcal{F}$ now to be the minimal face  of $\mathcal{S}^n_+$ containing the feasible region.
	Note that by 
	Observation~\ref{obs:exa}, any solution of 
	the SDP has rank at most $r$, and hence every matrix in $\mathcal{F}$ has rank at most $r$. 
	Consider now the line $L:=\{X+\lambda (Y-X): 
	\lambda\in\R\}$. Clearly $L$ is contained in the linear span of $\mathcal{F}$ and the line segment $L\cap\mathcal{F}$ is contained in the feasible region. Since $\mathcal{F}$ is pointed, the intersection $L\cap\mathcal{F}$ has at least one endpoint $Z$, necessarily lying in the relative boundary of $\mathcal{F}$. This matrix $Z$ therefore has rank at most $r-1$, a contradiction since $Z$ yields a 
	realization of $G$ in $\R^{r-1}$.
\end{proof}
\smallskip

In principle, one may now apply any off-the-shelf SDP 
solver to solve problem~\eqref{eqn:SDP_conv}. 
The effectiveness of such methods, however, 
depends heavily on the ``conditioning'' of the SDP 
system. 
In particular, if the system admits no feasible 
positive definite matrix, as is often the case 
(Observation~\ref{obs:exa}), 
then no standard method can be guaranteed to perform 
very well nor be robust to perturbations in the 
distance measurements. 

\subsection{Constraint mapping and the centering issue.}
To simplify notation, we will reserve some symbols for the mappings and sets appearing in formulations \eqref{eqn:main_eq} and  \eqref{eqn:SDP_conv}. To this end, define 
the mapping
$\KK : \Sn \rightarrow \Sn$ by  
\index{Lindenstrauss mapping, $\KK$}
\index{$\KK$, Lindenstrauss mapping}
\[
\KK(X)_{ij}:=X_{ii}+X_{jj}-2X_{ij}.
\]
The adjoint $\KK^*\colon\mathcal{S}^n\to\mathcal{S}^n$ is given by
 $$\KK^*(D)=2(\Diag(De)-D).$$
Moreover, the {\em Moore-Penrose pseudoinverse} of 
$\KK$ is easy to describe: 
for any matrix $D\in\mathcal{S}^n$ having all-zeros on 
the diagonal (for simplicity), we have
\[
\KK^{\dag}(D)=-\frac{1}{2}J \cdot D\cdot J,
\]
where $J:=I-\frac{1}{n}ee^T$ is the projection onto 
$e^{\perp}$. 
These and other related constructions have appeared in a 
number of publications; see for example 
\cite{homwolkA:04, MR97h:15032, hwlt91, MR2653818, MR2890931, MR2166851, MR2549047, MR1366579, MR2357790}.

Consider now the sets of {\em centered symmetric}, {\em centered PSD}, and {\em centered PSD low-rank} matrices 
\index{centered sets}
\index{$\Ss_c^n$, centered symmetric matrices}
\index{$\Ss_{c,+}^n$, centered PSD matrices}
\index{$\Ss_{c,+}^{n,r}$, centered PSD matrices of rank $\leq r$}
\begin{align*}
\Ss^n_{c}&:= \{X \in \Sn: Xe=0\},\\
\Ss^n_{c,+}&:= \{X \in \Sn_+: Xe=0\},\\
\Ss^{n,r}_{c,+}&:=\{X\in \Ss^n_{c,+}:\rank X\leq r\}. 
\end{align*}
Define now the coordinate projection $\PP\colon\mathcal{S}^n\to\R^E$ by setting $\PP(X)_{ij}=X_{ij}$.
In this notation, the feasible set \eqref{eqn:main_eq} can equivalently be written as $\{X\in\Ss^{n,r}_{c,+}: \PP\circ\mathcal{K}(X)=d\}$ while the relaxation \eqref{eqn:SDP_conv} is then $\{X\in\Ss^{n}_{c,+}: \PP\circ\mathcal{K}(X)=d\}$.

It is easy to see that $\Ss^{n}_{c,+}$ is a face of 
$\Snp$, and is linearly isomorphic to $\Ss^{n-1}_+$. 
Indeed, the matrix $ee^T$ exposes $\Ss^{n}_{c,+}$. 
More specifically, 
for any $n\times n$ orthogonal matrix $ \begin{bmatrix}
\frac{1}{\sqrt n}e & U  \cr
\end{bmatrix}$,
we have the representation
\begin{equation}\label{eqn:amb}
\Ss^{n}_{c,+}= U\mathcal{S}^{n-1}_+U.
\end{equation}
Consequently, we now make the following important convention:
the {\em ambient space} of $\Ss^{n}_{c,+}$ will always be 
taken as $\Ss^n_c$. 
The notion of faces of $\Ss^{n}_{c,+}$ and the 
corresponding notion of exposing matrices naturally 
adapts to this convention by appealing to \eqref{eqn:amb} 
and the respective standard notions for 
$\mathcal{S}^{n-1}_+$. Namely, we will say that $\mathcal{F}$ 
is a face of $\Ss^{n}_{c,+}$ if it has the form 
$\mathcal{F}=U\widehat{\mathcal{F}}U^T$
for some face $\widehat{\mathcal{F}}$ of 
$\mathcal{S}^{n-1}_+$, 
and that a matrix $Y$ exposes $\mathcal{F}$ whenever it 
has the form $U\widehat{Y}U^T$ for some matrix 
$\widehat{Y}$ exposing $\widehat{\mathcal{F}}$.

\section{Robust facial reduction for EDM completions.}
\label{sect:rfr}

In this section, we propose the use of robust facial reduction
for solving the least-squares formulation of the nonconvex EDM 
completion problem \eqref{eqn:main_eq}:
\begin{equation}\label{eqn:ranklss}
\begin{array}{ll}\displaystyle
\text{minimize} &\sum_{ij\in \EE} |X_{ii}+X_{jj}-2X_{ij} - d_{ij}|^2
\\
\text{subject to}& X\in\Ss_{c,+}^{n,r} .
\end{array}
\end{equation}
The main idea is to use the dual certificates arising from the 
rigid structures of the graph to construct a positive 
semidefinite matrix $Y$ of rank at least $n-r$, and then solve 
the convex optimization problem:
\begin{equation*}
\begin{array}{ll}\displaystyle
\text{minimize} &\sum_{ij\in \EE} |X_{ii}+X_{jj}-2X_{ij} - d_{ij}|^2
\\
\text{subject to}&X\in\Ss_{c,+}^n\cap Y^\perp.
\end{array}
\end{equation*}

Before describing our algorithmic framework for tackling 
\eqref{eqn:ranklss}, it is instructive to put it into context.
The authors of \cite{KH:2010} found a way to use 
the degeneracy of the system \eqref{eqn:main_eq} 
explicitly to design a combinatorial algorithm for 
solving \eqref{eqn:main_eq}, under reasonable conditions. 
%To describe the basic idea behind their procedure, 
%we first recall that a convex subset $\mathcal{F}$ 
%of the positive semidefinite cone $\mathcal{S}^n_+$ 
%is called a {\em face} if it can be written as
%\index{face}
%\begin{equation}\label{eqn:prim_form}
%\mathcal{F}=\left\{ 
%U\begin{bmatrix} A & 0 \\ 0 & 0 \end{bmatrix}U^T: 
%A\in\mathcal{S}^k_+
%\right\},
%\end{equation}
%for some $n\times n$ orthogonal matrix $U$ and some 
%integer $k\in\{0,1,\ldots,n\}$.
The authors observed that each 
$k$-clique in the graph $\GG$, with $k >r$, certifies 
that the entire feasible region of the convex
relaxation \eqref{eqn:SDP_conv} 
lies in a certain proper face $\mathcal{F}$ of the positive 
semidefinite cone $\mathcal{S}^n_+$. 
Therefore, the \emph{facial reduction} technique
of replacing $\mathcal{S}^n_+$ by the smaller set 
$\mathcal{F}$ can be applied on \eqref{eqn:SDP_conv} 
to obtain an equivalent problem involving fewer variables.
On a basic level, their method explores cliques in the graph, while 
possibly growing them, and intersects pairwise such 
faces in a computationally effective way. 

An important computational caveat of the facial reduction algorithm 
of \cite{KH:2010} is that the algorithm is highly unstable when 
the distance measurements are corrupted by noise---a ubiquitous 
feature of the EDM completion problem in virtually all applications. 
The reason is simple: randomly perturbed faces of the 
semidefinite cone typically intersect only at the origin. 
Hence small perturbations in the distance measurements 
will generally lead to poor guesses of the face intersection 
arising from pairs of cliques. Moreover, even if pairs of 
cliques can robustly yield some facial information, 
the accumulated error compounds as the algorithm moves from 
clique to clique.
Remarkably, we show that this difficulty can be overcome by using ``dual'' 
representations of faces to aggregate the noise.  Indeed, the salient feature of the dual representation is that 
it is much better adapted at handling noise. 
%To this end, consider a face $\mathcal{F}$ of $\mathcal{S}^n_+$ 
%in the ``primal'' form \eqref{eqn:prim_form}.
%Subdivide $U$ now into two parts 
%$U=\begin{bmatrix} U_R & U_N  \end{bmatrix}$, 
%where $U_R$ has $k$ columns and $U_N$ has $n-k$ columns.
%Then $\mathcal{F}$ can equivalently be written as 
%\begin{equation}\label{eqn:dual_form}
%\mathcal{F}= (U_N \cdot B \cdot U^T_N)^{\perp}\cap 
%\mathcal{S}^n_+,
%\end{equation}
%for any nonsingular matrix $B$ in $\mathcal{S}^{n-r}_+$, 
%where $\perp$ denotes the orthogonal complement with 
%respect to the trace inner product. 
%Going back and forth between primal 
%\eqref{eqn:prim_form} and dual \eqref{eqn:dual_form} 
%representations is easy, 
%requiring a single eigenvalue decomposition, 
%or more cheaply a $QR$ factorization, for example. 
%In general, if a face has the form 
%$\mathcal{F}=Y^{\perp}\cap \mathcal{S}^n_+$ 
%for some positive semidefinite matrix $Y$, 
%then we say that $Y$ {\em exposes} $\mathcal{F}$. 

Before proceeding with the details of the proposed 
algorithmic framework, we provide some intuition.
To this end, an easy computation shows that if
$Y_i$ exposes a face $\mathcal{F}_i$ of $\mathcal{S}^n_+$ 
(for $i=1,\ldots,m$), 
then the sum $\displaystyle\sum_{i}Y_i$ exposes the intersection 
$\displaystyle\bigcap_{i}\mathcal{F}_i$.
Thus the faces $\mathcal{F}_i$ 
intersect trivially if and only if the sum 
$\displaystyle\sum_{i}Y_i$ is positive definite. On the other hand, 
if the true exposing vectors 
arising from the cliques are corrupted by noise, 
then one can round off the small eigenvalues of 
$\displaystyle\sum_{i}Y_i$ (due to noise) to guess at the true intersection 
of the faces arising from the noiseless data.

\subsection{The algorithmic framework.}
To formalize the outlined algorithm, we will need the following basic result, which
in a primal form was already the basis 
for the algorithm in \cite{KH:2010}. The dual form, however, is essential for our purposes.
For easy similar alternative proofs, see  
\cite[Theorem 4.9]{coor} and \cite[Theorem 4.1]{krislock:2010}.
Henceforth, given a clique $\alpha\subseteq\VV$ (meaning, a subset of vertices such that every two are adjacent), we use $d_{\alpha}\in\Ss^{|\alpha|}$
to denote the symmetric matrix formed from 
restricting $d$ to the edges between
the vertices in $\alpha$.

\smallskip
\begin{thm}[One clique facial reduction] %for noiseless SNL]
	\label{thm:EDM_face} %{\hfill \\ }
%	Consider a noiseless SNL instance on the graph $\GG=(\VV,\EE)$ 
%	with distance measurements $d\in\R^{\EE}$.
	Suppose that the subset of vertices $\alpha:=\{1,\ldots,k\}\subset\VV$ 
    is a clique in $G$.
%	Let $d_{\alpha}\in \mathcal{S}^k$ be the  
%	restriction of $d$ to $\alpha$, and 
Define the set 
	$$
	\begin{array}{rcl}
	\widehat{\Omega} &:= &
	\{X\in \Ss^n_{c,+}:  [\KK(X)]_{ij}=d_{ij} 
	\quad\textrm{ for all }\quad 1\leq i< j\leq k\}.       
	%        \textrm{ for all } 1\leq i<j\leq n\}
	\end{array}
	$$
%	of all the feasible solutions of \eqref{eqn:SDP_conv}.
	Then for any matrix $\widehat{Y}$ exposing 
	$\face\big(\KK^{\dag}d_{\alpha};\Ss^k_{c,+} \big)$,
	$$
	\textrm{the matrix}\ \ 
	\begin{bmatrix}
	\widehat{Y} & 0  \cr \\
	0 & 0
	\end{bmatrix}       
	\ \  \textrm{ exposes }\ \  \face(\widehat{\Omega};\Ss^n_{c,+}).
	$$
\end{thm}

In particular, under the assumptions of the theorem, the 
entire feasible region of \eqref{eqn:SDP_conv} is contained  
in the face of $\Ss^n_{c,+}$ exposed by $\begin{bmatrix}
	\widehat{Y} & 0  \cr \\
	0 & 0
\end{bmatrix}$. The assumption that the first $k$ vertices formed a 
clique is of course made without loss of generality.
%Thus any clique in the graph $\GG$ certifies that the 
%entire feasible region of the convex relaxation 
%\eqref{eqn:SDP_conv} to the noiseless SNL problem is 
%contained in a certain face of the cone $\Ss^n_{c,+}$. 
%If the clique has more than $r$ vertices, then the 
%corresponding face is guaranteed to be proper, 
%meaning that it is strictly contained in $\Ss^n_{c,+}$. 
We can now state our proposed algorithmic framework, 
in Algorithm \ref{alg_main} below.
%the noisy SNL problem, which works for general
%noisy Euclidean distance matrix (EDM) completion problem. 
%(Recall that $D\in\Sn$ is a \emph{Euclidean distance matrix} 
%if there exist vectors $x_1,\ldots, x_n$ of the same length 
%such that $D_{ij}=\|x_i-x_j\|_2^2$ for all $i,j$, 
%or equivalently, if $D\in\KK(\Ss^n_+)$.)
\index{Euclidean distance matrix, EDM}
\index{EDM, Euclidean distance matrix}

\begin{algorithm}
	\caption{Basic strategy for EDM completion
		\label{alg_main}}
	\begin{algorithmic}
		\STATE{\textbf{{INPUT}:} 
			A weighted graph $\GG=(\VV,\EE,d)$, 
%			a vector $d\in\R^{\EE}$, 
            and a target rank $r\geq 0$;} 
%		\STATE{\textbf{{OUTPUT}:} a Gram matrix $X$;}
		\STATE{\textbf{{PREPROCESSING}}}: 
		\begin{enumerate}
			\item Generate a set of cliques $\Theta$ in $G$; 
		\item Generate a set of weight functions 
			$\{\omega_{\alpha}\colon{\R}^{\EE}\to{\R}_+\}_{\alpha\in\Theta}$;
		\end{enumerate}
		\FOR{each clique $\alpha$ in $\Theta$}{
			\STATE{$k\leftarrow$ $|\alpha|$;}
			\STATE{$X_{\alpha} \leftarrow$ a nearest matrix in 
				$\Ss^{k,r}_{c,+}$ to $\KK^{\dag}d_{\alpha}$;}
			\STATE{$W_{\alpha} \leftarrow$ exposing vector 
				of $\face(X_{\alpha}, \Ss^{k}_{c,+})$ extended to $\Ss^n$ by padding zeros;}
%			\STATE{$W_{\alpha} \leftarrow$ extension of $\widehat{W}_{\alpha}\in\Ss^k$ to 
%			a matrix in $\Ss^n$ (by padding zeros);
%			}
		} \ENDFOR
		\STATE{$W\leftarrow \sum_{\alpha\in\Theta} 
			\omega_{\alpha}(d) \cdot W_{\alpha}$;}
		\STATE{$Y\leftarrow$ a nearest matrix in $\Ss^{n,n-r}_{c,+}$ to $W$;}
		\STATE{$X\leftarrow$ a solution of }
		\begin{equation}\label{eqn:small_sys}
		%\label{eq:startminxrankr}
		\begin{array}{lll}
		\text{minimize} &   \|\PP\circ\KK(X) - d\| \\
		\text{subject to}   &    X\in  Y^{\perp}\cap\Ss^{n}_{c,+};
		\end{array}
		\end{equation}
		\RETURN $X$;
	\end{algorithmic}
\end{algorithm}

%\subsection{Robustness guarantees}

%The proposed framework is provably robust. 

Some comments are in order. First, there is great flexibility in the 
preprocessing stage, and it will be described in Subsection~\ref{sect:clique_set}. 
Secondly, finding ``nearest matrices'' in $\Ss^{k,r}_{c,+}$ and in 
$\Ss^{n,n-r}_{c,+}$ is easy as a result of the Eckart-Young theorem. 
The details are worked out in Appendix~\ref{sect:nearest_point}. 
Solving the small dimensional least squares problem \eqref{eqn:small_sys} 
is also standard. We discuss it in Appendix~\ref{sect:lss}. 
In fact, very often (under the assumptions of Theorem~\ref{thm:robust} 
below) the linear least squares solution of 
%\begin{equation}\label{eqn:small_sys}
%\label{eq:startminxrankr}
$
\displaystyle{\min_{X\in\mathcal{V}}}\,   \|\PP\circ\KK(X) - d\|
$
%\end{equation}
already happens to be positive definite, where $\mathcal{V}$ denotes the linear span of the face $Y^{\perp}\cap\Ss^{n}_{c,+}$. Hence this step typically does not require any optimization solver to be invoked. Indeed, this is a direct consequence of the rudimentary robustness guarantees of the method, outlined in Appendix~\ref{sec:robustness}.

%We now outline a rudimentary robustness guarantee of  Algorithm \ref{alg_main}.

\subsection{Implementing facial reduction for noisy EDM.}

In the following, we elaborate on some of the main
ingredients of Algorithm \ref{alg_main}: 
\begin{itemize}
	\item
	the choice of the clique set $\Theta$ and weight functions 
	$\{\omega_\alpha\}_{\alpha\in\Theta}$ 
	(in Section \ref{sect:clique_set});
	\item 
	the nearest-point mapping to $\Ss^{k,r}_{c,+}$
	(in Appendix \ref{sect:nearest_point}); and
	\item 
	the solution of the least squares problem
	\eqref{eqn:small_sys}
	(in Appendix \ref{sect:lss}).
\end{itemize}
%To improve the solution quality of Algorithm \ref{alg_main},
%we also use preprocessing and postprocessing steps.
%The preprocessing involves refining the clique set $\Theta$
%and ensures that the exposing vector $U$ computed in 
%Algorithm \ref{alg_main} has exactly $r+1$ eigenvalues that 
%are very close to zero;
%see Section \ref{sect:preprocessing}.
%For postprocessing, we use a nonlinear optimization technique
%to find a local solution of \eqref{eqn:ranklss} using the solution
%$X$ from Algorithm \ref{alg_main} as the initial point. 
%Such a \emph{local refinement} procedure, which by itself
%often fails to find a global optimal solution, can greatly improve
%the solution quality of Algorithm \ref{alg_main}.
To improve the solution quality of Algorithm \ref{alg_main},
we perform a postprocessing \emph{local refinement} step:
we use the solution $X$ from Algorithm \ref{alg_main} as an 
initial point for existing nonlinear optimization methods
to find a local solution of \eqref{eqn:ranklss}. 
While general nonlinear optimization methods often fail
to find a global optimal solution, when used as a local refinement
procedure they can greatly improve
the solution quality of Algorithm \ref{alg_main}.

\subsubsection{Choosing the clique set and the weights.}
\label{sect:clique_set}

We first discuss the choice of the clique set $\Theta$, 
which is crucial for the success of 
Algorithm \ref{alg_main}, since the exposing vector $Y$ 
is formed based on the clique information.
The level of rigidity of the graph known to Algorithm \ref{alg_main}
is determined by the clique set $\Theta$: if insufficiently many cliques are
present, then the estimate of the exposing vector $Y$ will likely be poor.

In practice, it is inefficient to compute the set of 
\emph{all} cliques of $\GG$ (noting that 
determining whether a graph has a clique of an arbitrary 
given size is NP-hard),
so we can only hope to find a subset of cliques of the graph.
%Alternative to the common symrcm approach (CITATION?), we 
We apply a simple brute-force subroutine 
on the adjacency matrix $H$ of the given graph
to find a collection $\Theta$ of cliques, as in Algorithm
\ref{algo:choosingCliques} below.

\begin{algorithm}
\caption{Finding a collection of cliques in a graph\label{algo:choosingCliques}}
\begin{algorithmic}
\STATE{\textbf{(INPUT) }
A simple graph $\GG=(\VV,\EE)$, 
integer $\bar k\geq2$;}
\STATE{\textit{Step 1: constructing $\Theta_1$ and $\Theta_2$}}
\STATE{$\Theta_1\leftarrow \emptyset$;}
\FOR{each vertex $v=1,\ldots,n$}{
  \STATE{find a subset $\alpha_v$ of neighbors of $v$ that forms a 
clique;}
  \STATE{$\Theta_1\leftarrow \Theta_1\cup\{\alpha_v\}$;}
} \ENDFOR
  \STATE{$\Theta_2\leftarrow\{uw\in \EE: \nexists\, \alpha\in\Theta_1
  \,\textrm{ such that }\,u,w\in \alpha\}$;}
\STATE{\textit{Step 2: constructing $\Theta_3,\ldots,\Theta_{\bar k}$}}
\FOR{$k=3,\ldots,\bar k$}{
  \STATE{$\Theta_k\leftarrow\emptyset$;}
  \FOR{each $\alpha\in\Theta_{k-1}$}{
    \IF{all the vertices in $\alpha$ share a common neighbor $v$}
    {
      \STATE{$\Theta_k\leftarrow \Theta_k \cup \{\alpha\cup \{v\}\,\}$;}
    } \ENDIF
  } \ENDFOR
} \ENDFOR
\STATE{$\Theta \leftarrow \bigcup_{k=1}^{\bar k} \Theta_k$\;}
\STATE{\textbf{(OUTPUT) }$\Theta$, a set of cliques in $\GG$
of size up to $\bar k$.}
\end{algorithmic}
\end{algorithm}

%First, for each vertex $v$, we use a greedy heuristic 
%to find a sufficiently large all-ones principal minor
%of $H(\delta_v,\delta_v)+I$ 
%(where $\delta_v:=\bigcup\{u\in \VV: uv\in \EE\}$), 
%which would give a clique covering $v$; let $\Theta_1$
%be the set of distinct cliques found in this manner. 
%Then we use a brute-force method to find a collection 
%$\Theta_k$ of 
%$k$-cliques $\alpha$ of $\GG$ such that $\alpha$ 
%is not contained in any clique in 
%$\Theta_1$, for $k=2$ up to a user-defined
%maximum size: for $k=2$ we collect all the edges
%that are not in any clique in $\Theta_1$, 
%and for each $k\geq3$, we consider 
%each of the $k-1$ cliques and determine if the vertices
%in that clique have common neighbors.
%This results in a collection of cliques 
%$\Theta:=\bigcup_k\Theta_k$.

In the first step, computing $\Theta_1$ is
very fast since it involves only repeatedly removing
rows and columns of $H(\delta_v,\delta_v)+I$ 
that contain zero for each vertex $v$
(where $H$ is the $\{0,1\}$-adjacency matrix of $\GG$ 
having zero diagonal and $\delta_v$ is the set of neighbors of 
$v$). As for the second step,
while the brute-force method of listing all
cliques of fixed sizes would be prohibitive in practice, 
we find that the restriction imposed by $\Theta_1$
cuts down a huge number of smaller non-maximal cliques 
that we need to keep track of, 
and the use of $\Theta_1$ significantly 
speeds up the second step.

Algorithm \ref{algo:choosingCliques} provides a very basic 
clique-selection framework.
When working with a particular application (e.g., the 
sensor network localization problem, described in 
Section \ref{sect:numerics} below), the robustness of 
Algorithm \ref{alg_main} can be improved significantly 
when the clique selection process is specialized for 
that application. 

Now we discuss the weight functions $\{\omega_{\alpha}\}_{\alpha
\in\Theta}$. In Algorithm \ref{alg_main},
we do not treat each clique in $\Theta$ equally, 
given that the noise in the distance measurements does not have 
to be uniform and it may not be possible to recover all the 
cliques with the same level of accuracy. 
We gauge the amount of noise present in the distance measurements
of cliques as follows:
for each clique $\alpha\in\Theta$, 
as before letting $d_\alpha\in\mathcal{S}^{\alpha}$
be the restriction of the distance measurements $d$ to the clique,
we estimate the noise present in $d_\alpha$ by considering 
the eigenvalues of $\mathcal{K}^\dagger d_\alpha$:
\begin{equation}
\label{eq:noisedef}
\noise_\alpha(d) := \frac{ \sum_{j=1}^{|\alpha|-r} 
	\lambda^2_j(\mathcal{K}^{\dagger} d_\alpha)}
{0.5|\alpha|(|\alpha|-1)}.
\end{equation}
Here $\lambda_j(\mathcal{K}^{\dagger} d_\alpha)$ refers to the $j$'th smallest eigenvalue of the matrix $\mathcal{K}^{\dagger} d_\alpha$. 
The value $\nu_{\alpha}(d)$ is the scaled squared-$\ell_2$ norm 
of the violation of the rank constraint in the clique $\alpha$, 
i.e., the constraint $\rank(\mathcal{K}^\dagger d_\alpha )\leq r$.
In the case where no noise is present in the distance measurements
$d$, we have $\noise_\alpha(d)=0$ since the matrix 
$\mathcal{K}^\dagger d_\alpha\in\mathcal{S}^{|\alpha|}_+$ 
is of rank at most $r$.
To each clique $\alpha$, we assign the weight
\[
\omega_\alpha(d)
:= 
1- \frac{ \noise_\alpha(d)}
{\sum_{\beta\in\Theta} \noise_\beta(d)}.
\]
This choice of weight reflects the 
contribution of noise in the clique $\alpha$ to the total noise
of all cliques (where the noise is measured by \eqref{eq:noisedef}).
If a clique $\alpha$ is relatively noisy compared to 
other cliques in $\Theta$ or contains an outlier, 
the weight $\omega_{\alpha}(d)$ would be 
smaller than $\omega_{\beta}(d)$ for most $\beta\in\Theta$.

\subsubsection{Postprocessing: local refinement.}
\label{sect:postprocessing}

Following Algorithm \ref{alg_main}, 
we implement a \emph{local refinement}, 
which could greatly improve the solution quality.
By local refinement, we mean the use of a nonlinear 
optimization algorithm
for solving the nonconvex problem \eqref{eqn:ranklss} 
(which has a lot of local minima) using
the output of Algorithm \ref{alg_main} as the initial point.
Local refinement has been commonly used for SDP-based algorithms for SNL
problems and noisy EDM completion problem; see \cite{BYACM:06, BYSJC:08}.
\index{local refinement}

For local refinement, 
we use the steepest descent subroutine from the SNL-SDP 
package \cite{BYIEEE:06}. Suppose that $X^*=P^*(P^*)^T$
is the solution of 
\eqref{eqn:small_sys} found at the end of Algorithm
\ref{alg_main}. We use $P^*$ as an initial point for
the steepest descent method to solve 
the nonlinear optimization problem
\begin{equation}
\label{eqn:ranklss2}
   \min_{P\in\R^{n\times r}}
   \|\PP\circ\mathcal{K}(PP^T)-d\|^2.
\end{equation}
By itself, the steepest descent method usually fails to 
find a \emph{global} optimal solution of \eqref{eqn:ranklss2}
and instead gets trapped at one of the many critical points,
since the problem is highly nonconvex. 
On the other hand, we observe that Algorithm \ref{alg_main} can produce excellent initial points for such nonlinear optimization schemes. 

%when combined with Algorithm \ref{alg_main} the method can be very effective.
%
%\texttt{Manopt}, a MATLAB package for 
%optimization on manifolds (also known as Riemannian optimization)
%\cite{JMLR:v15:boumal14a},
%which is well-suited for solving rank-constrained optimization problems like \eqref{eqn:ranklss}.
%The Riemannian optimization techniques have been previously used
%to solve the \emph{noiseless} Euclidean distance matrix completion 
%problem \cite{mishra:low-rank}.
%Specifically, we use a Riemannian trust region method to solve
%\index{\texttt{Manopt}}
%\begin{equation}
%\label{eqn:ranklss2}
%\min_{P\in\R^{(n-1)\times 2}}
%\|\PP\circ\widehat{\mathcal{K}}(PP^T)-d\|_2^2
%\end{equation}
%(which is equivalent to \eqref{eqn:ranklss}).
%By itself, the Riemannian trust region method, like
%most nonlinear optimization techniques, usually fails to 
%find a \emph{global} optimal solution of \eqref{eqn:ranklss2}
%and instead gets trapped at one of the many critical points
%(as we can see in Table \ref{table1} in 
%Section \ref{sect:numerics} on numerical results),
%since the problem is highly nonconvex. Yet by providing 
%the trust region with a ``good'' initial point (obtained
%from Algorithm \ref{alg_main}), we can greatly improve the 
%performance of the Riemannian trust region method applied on 
%\eqref{eqn:ranklss2}, as illustrated in the numerical results.

\subsection{Application on the sensor network localization
problem.}\label{sect:numerics}

\renewcommand{\arraystretch}{1.3}

In this section, we apply robust facial reduction
(Algorithm \ref{alg_main}) on %andomly generated instances of the
the anchorless \emph{sensor network localization} (SNL) 
problem in $\mathbb{R}^2$. 
The task is to locate $n$ wireless sensors in $\mathbb{R}^2$, given 
the noisy squared Euclidean distances between sensors that are within 
a given radio range of each other.
% and the positions of some of the sensors (called anchors)
Semidefinite programming techniques have been used extensively for the SNL problem; see for example 
\cite{BYACM:06,BYIEEE:06,BYSJC:08,BY:04,by03,KH:2010,sy07,pt09,wzyb08}. 

One important characteristic of the SNL problem is 
the presence of a radio range:
the distance between two sensors is available if and only if
the distance is no larger than the radio range.
%The preprocessing step in the robust facial reduction can 
%be specialized for the SNL problem. 
%In SNL, all the unknown distances are no less than the radio
%range. 
This simple feature allows us to specialize the
preprocessing step of finding the clique set $\Theta$ 
in the basic robust facial reduction 
framework, using the 
\emph{nonrigid clique union technique} from \cite{KH:2010}: 
we refine the feasible region in the EDM completion
problem by removing solutions that violate the implicit
constraints imposed by the radio range.
This step approximately completes the partial EDM \emph{locally}, 
while generating a set of larger cliques and 
reducing the error in calculating the exposing vector 
using the noisy distance measurements.

%----------------------------------------------------------
\subsubsection{Preprocessing via clique union.}\label{subsubsec:clique_union}

The first step is to determine an ordering of the cliques
$\Theta = \{\alpha_1,\ldots,\alpha_{|\Theta|}\}$
such that
\begin{equation}
\label{eq:clique_union}
\text{
	$\alpha_{j-1}$ and $\alpha_{j}$ intersect 
	in at least 2 vertices,
}\ \forall\, j.%=2,\ldots,|\Theta|.
\end{equation}
Such an ordering can be found using a greedy approach:
start with the largest clique $\alpha_1\in\Theta$ and 
$\widehat\Theta:=\Theta$,
and for each $j\geq2$, 
pick $\alpha_j$ 
among all the cliques in $\widehat\Theta$ intersecting 
with $\alpha_{j-1}$ in at least 2 vertices, the one that maximizes
the set difference $|\alpha_j \backslash \alpha_{j-1}|$;
then update $\widehat\Theta$ by removing from it 
all the cliques whose nodes are covered by 
$\bigcup_{l=1}^j \alpha_l$.

If the graph $\GG$ is sparse, such an ordering may not exist.
Nonetheless, as long as $\GG$ does not have a cut vertex,
it is possible to cover all the vertices of $\GG$
with multiple sequences of cliques, each satisfying the 
condition \eqref{eq:clique_union}.
(Note that if a noiseless SNL instance is uniquely $r$-localizable, in the sense of \cite{SoYe:05},
for any $r\geq2$, then the corresponding graph cannot have a cut vertex.)

%In the following we discuss the 
%\emph{clique union} technique under the assumption that
%an ordering of the elements of $\Theta$ satisfying
%\eqref{eq:clique_union} exists.
%; the technique can be easily extended 
%even when such an assumption does not hold.

Suppose that we have found an ordering of the elements of 
$\Theta$ satisfying the condition \eqref{eq:clique_union}.
Then we would perform a sequential clique union procedure,
whose goals are to ensure that the matrix $U$
found in Algorithm \ref{alg_main} is not too far from
$\Ss^{n,n-r}_{c,+}$, and to avoid errors arising from 
(nearly) nonrigid intersection, which we illustrate in 
the following example.

\smallskip
\begin{example} 
{\rm	\label{eg:1}
	Suppose that we have 5 sensors with radio range $0.05$,
	whose true locations are given by
	\[
	P=\begin{bmatrix}
	0.4582 &  0.4793 &  0.5031 &  0.4360 &  0.4467 \\
	-0.4116 & -0.3952 & -0.3221 & -0.3150 & -0.3393
	\end{bmatrix}^T.
	\]
	Then the corresponding graph is as in
	the left picture in Figure \ref{fig:example}:
	only the distance between sensors 1 and 5 is missing.
	The graph has two cliques $\alpha_1=\{1,2,3,4\}$ and
	$\alpha_2=\{2,3,4,5\}$. Sensors $2,3,4$ in the clique intersection
	are almost collinear; $\alpha_1$ and $\alpha_2$ almost intersect
	\emph{nonrigidly} (locally), in the sense that the realization
	of sensor locations:
	\[
	\tilde P=
	\begin{bmatrix}
	0.4582 &  0.4793 &  0.4360 &  0.4467 &  0.4051
	\\
	-0.4116 & -0.3952 & -0.3150 & -0.3393 & -0.3750
	\end{bmatrix}^T
	\]
	obtained by reflecting $\alpha_1$ along the
	line passing through vertices 2 and 3
    (in the center of Figure \ref{fig:example})
	would give almost the same partial EDM:
	\[
	   \left\|
          \mathcal{P}\bigl(\mathcal{K}(PP^T)-
          \mathcal{K}(\tilde P\tilde P^T)\bigr)\right\|
	\approx
	6.84\times 10^{-4},
	\]
	where $\mathcal{P}:\Ss^n\to\R^{\EE}$ is the canonical projection.
	%(See the middle picture in Figure \ref{fig:example} for the sensor locations
	%depicted by $\tilde P$.)
	In the presence of uncertainty in distance measurements,
	both $P$ and  $\tilde P$ seem to be
	reasonable realization of the sensor locations.
	Yet with the additional knowledge of the radio range, we know
	that it is unlikely $\tilde P$ gives the approximate sensor locations, since that would mean sensors 1 and 5 are in each other's radio
	range.

	\begin{figure}[!ht]
		\centering
		\caption{\small
			Left: true location of 5 sensors (given by $P$), 
			with edges indicating known distances.
			Center: realization of sensor location (given by $\tilde P$) 
			satisfying the known distances but 
			violating the radio range.
			Right: solution from Algorithm \ref{alg_main}.
			Circles {\textcolor{blue}{$\circ$}}\,: 
               clique $\alpha_1$; 
            pluses {\scriptsize\textcolor{red}{+}\,}: 
               clique $\alpha_2$.
			\label{fig:example}
		}
		\epsfig{file=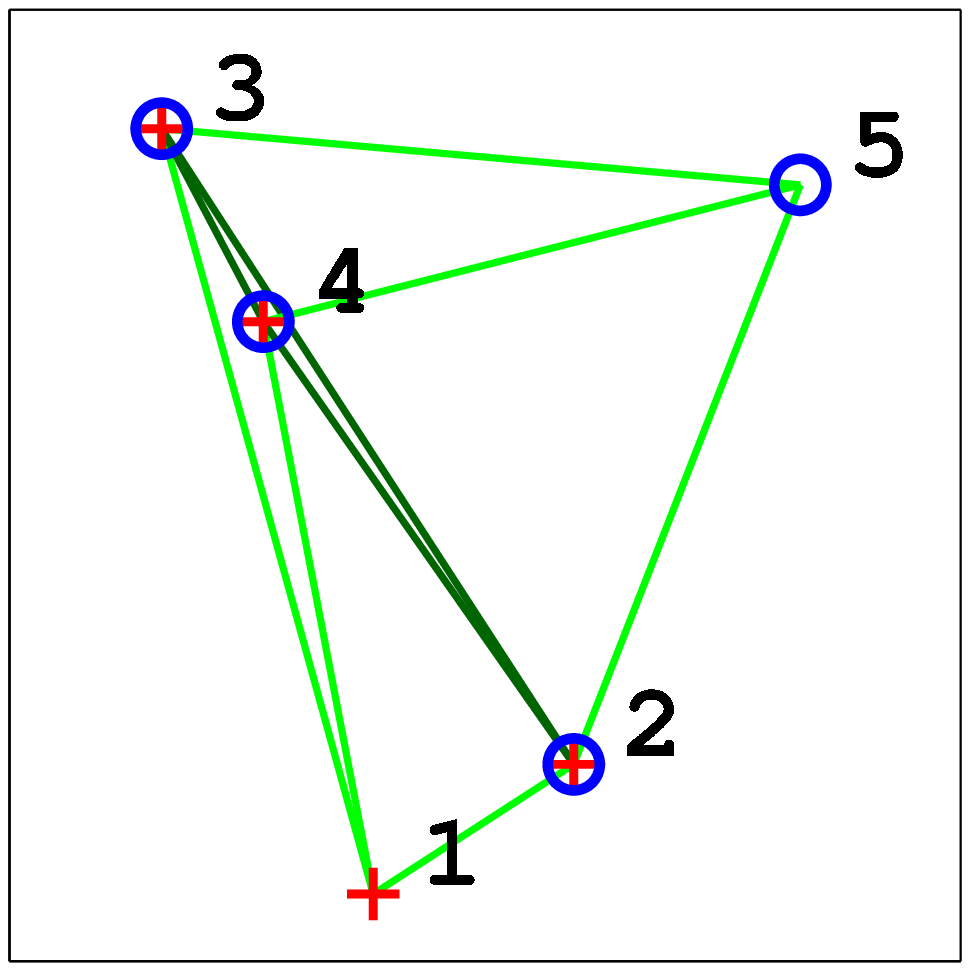, height=1.65in, width=1.65in}\,
		\epsfig{file=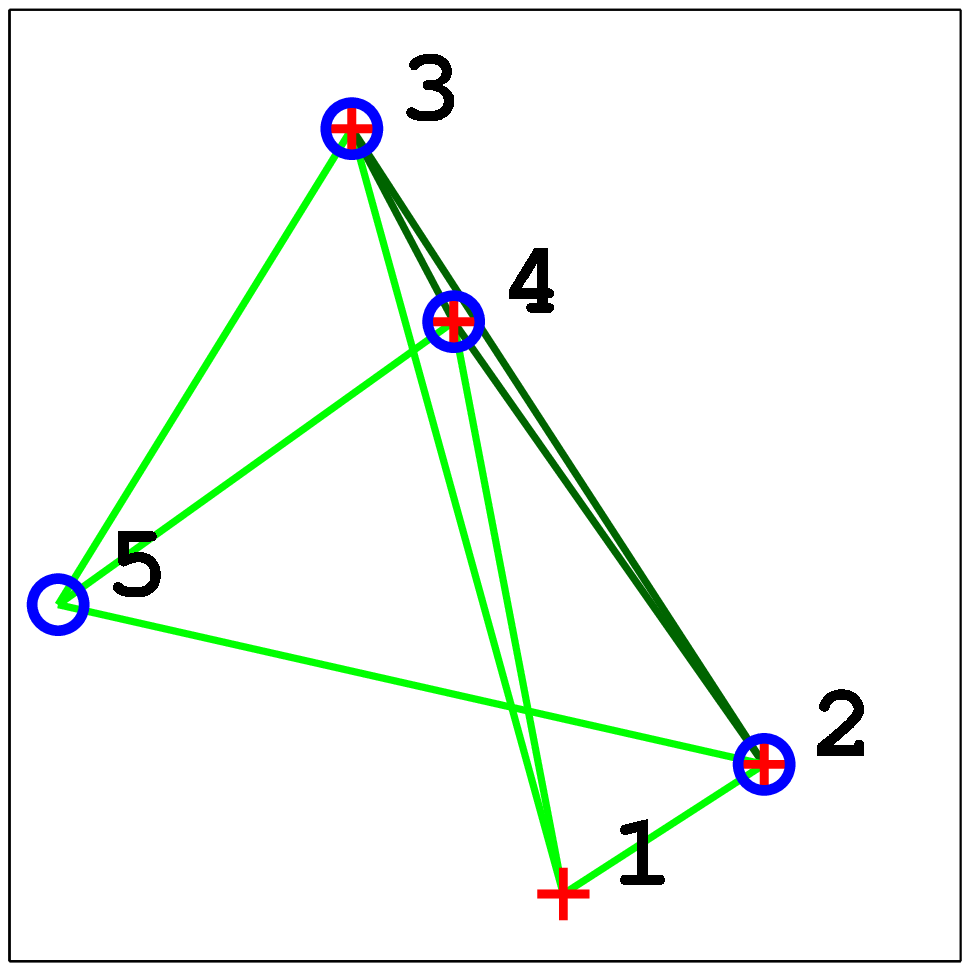, height=1.65in, width=1.65in}\,
		\epsfig{file=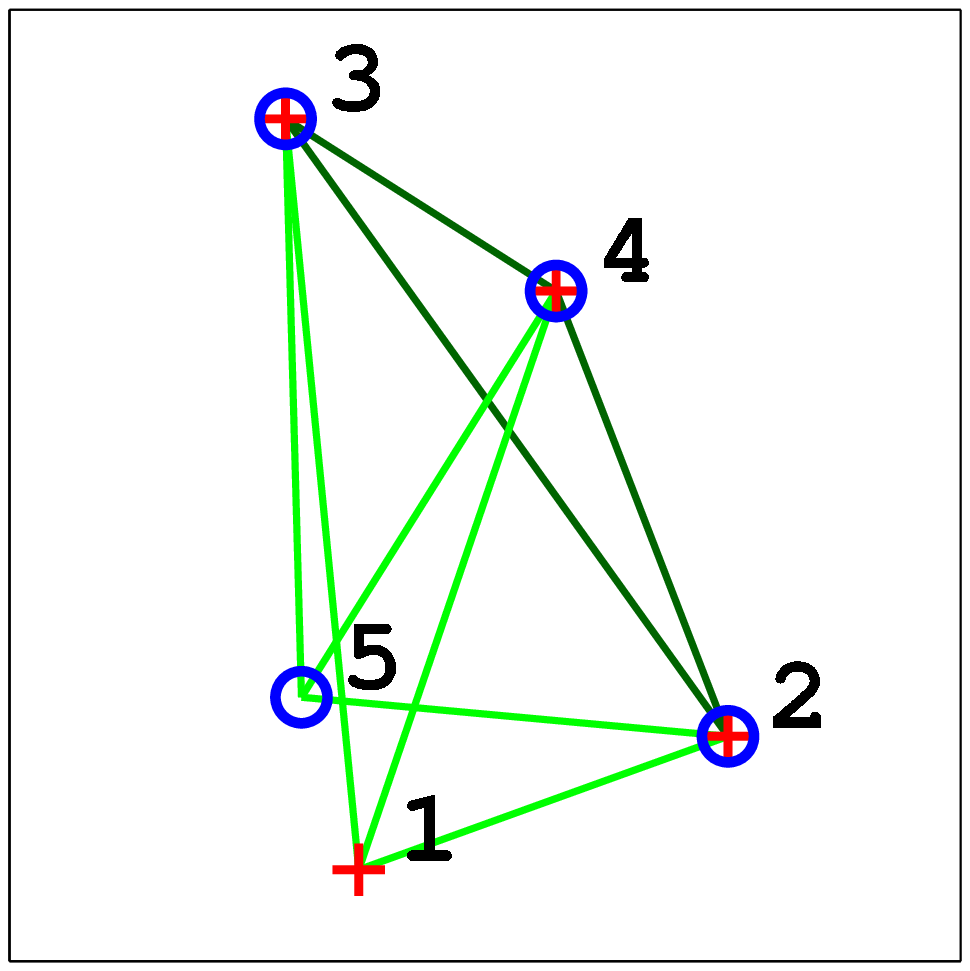, height=1.65in, width=1.65in}
	\end{figure}

	Now suppose that the distance measurements are corrupted with 5\%
	Gaussian noise (see the multiplicative noise model outlined in 
	Algorithm \ref{alg:multinoise}).
	If we apply Algorithm \ref{alg_main} on the noisy input, then
	the realization of sensor locations could be as in the right picture in
	Figure
	\ref{fig:example}, where sensors 1 and 5 are much closer
	than they should be.
	%(since such a realization as in Figure \ref{fig:example-1c} would 
	%mean that sensors 1 and 5 are in each other's radio range).
	Note that
	the right picture in Figure \ref{fig:example}
	shows a minor perturbation
	of the ``incorrect'' realization $\tilde P$.}
\end{example}

\smallskip
Scenarios depicted in Example \ref{eg:1} can be quite prevalent:
two of the cliques in $\Theta$ may intersect (almost) nonrigidly (locally), 
and there would be two localizations that
give similarly good least squares solutions, corresponding to two 
different ``reflections''. 
To ensure the robustness of the facial reduction algorithm,
we perform a clique union on $\alpha_{j-1}$ and $\alpha_j$ 
(for each $j=2,\ldots,|\Theta|$), by using a Procrustes rotation
to match the cliques $\alpha_{j-1}$, $\alpha_j$ at the intersection
and to ensure also that the unknown distances calculated are not much
smaller than the radio range.
This constitutes a \emph{local EDM completion}: this approach
localizes the two cliques $\alpha_{j-1}$ and $\alpha_j$, and
as a result we obtain the distances between all the vertices in 
$\alpha_{j-1}\cup\alpha_j$. 
After we obtain a realization of $\alpha_{j-1}\cup\alpha_j$,
we use that realization to compute an exposing vector corresponding
to $\alpha_{j-1}\cup\alpha_j$.
This preprocessing step results in exposing vectors for the larger cliques
$\beta_j=\alpha_{j}\cup\alpha_{j+1}$ for $j=1,\ldots, |\Theta|-1$. 
The larger cliques
$\beta_1,\ldots,\beta_{|\Theta|-1}$ intersect at more vertices, 
lowering both the possibility that some of the clique intersections 
are nonrigid and the error of the exposing vector calculation. 

%----------------------------------------------------------
\subsubsection{Numerics.}
%----------------------------------------------------------

For the numerical tests, we generate random instances of the SNL 
problem based on a \emph{multiplicative noise model} 
(\cite{BYACM:06, BYIEEE:06}) outlined
in Algorithm \ref{alg:multinoise}.
\index{multiplicative noise model}
\begin{algorithm}
	\caption{Multiplicative noise model
		\label{alg:multinoise}}% for problem instance generation}
	\begin{algorithmic}
		\STATE \textbf{INPUT: }
		\# sensors $n$,
		noise factor $\sigma$,
		radio range $R$;
		\STATE For each $i,j=1,\ldots,n$:
		\STATE \ - pick i.i.d. $p_i\in[-0.5,0.5]^{2}$ 
		with uniform distribution
		\STATE \ - pick i.i.d. $\epsilon_{ij}\in\mathcal{N}(0,1)$
		(standard normal distribution)
		%, $\forall\, i,j=1,\ldots,n$;
		\STATE Compute $D\in\Ss^n$ by
		\[
		D_{ij} = (1+\sigma\epsilon_{ij})^2\|p_i-p_j\|^2,
		\ \ \forall\, i,j=1,\ldots,n;
		\]
		\STATE \vspace{-1em} Build graph $\GG=(\{1,\ldots,n\}, \EE)$, where
		\[
		ij\in \EE \iff \|p_i-p_j\|\leq R;
		\]
		\STATE \vspace{-0.8em} $d\leftarrow [D_{ij}]_{ij\in \EE,\, i<j}\in\R^{\EE}$;
		\STATE  \textbf{OUTPUT: }%true positions $p_1,\ldots,p_n$,
		noisy distance measurements $d\in\R^{\EE}$ and graph $\GG$.
	\end{algorithmic}
\end{algorithm}

Since the instances generated by the multiplicative noise model 
come with the true sensor locations,
we can gauge the performance of the robust facial reduction on
random instances from the multiplicative noise model using 
the \emph{root-mean-square deviation} (RMSD).
Suppose that the true centered locations of the sensors
are stored in the rows of the matrix $P\in\R^{n\times 2}$, 
and $X\in\Ss^{n,r}_{c,+}$ is the output of Algorithm \ref{alg_main}.
Then $X=\tilde P \tilde P^T$ for some $\tilde P\in\R^{n\times 2}$,
whose rows store the estimated centered locations.
%We assume that a \emph{Procrustes rotation} is performed on 
%$\tilde P$, i.e., $I\in \arg\min\{\|\tilde P U - P \| : 
%U^TU=I, U\in\R^{r\times r}\}$,
%so that the estimated locations given by the rows of $\tilde P$
%are matched to those of $P$ by a rotation.
The RMSD of the estimated
$\tilde P$ relative
to the true centered locations $P$ is defined as:
\index{root mean standard deviation, RMSD}
\index{RMSD, root mean standard deviation}
\begin{equation}\label{eqn:rmsd}
   RMSD := 
   \min\left\{\frac{1}{\sqrt{n}} \|\tilde P U - P\|_F\ \colon
      \ U^TU=I,\ U\in\R^{r\times r}
   \right\}.
\end{equation}
%measuring the distances between the true locations and the 
%estimated locations subject to a rotation.

A typical output of Algorithm \ref{alg_main} applied 
on an instance generated by the multiplicative noise 
model is illustrated in Figure \ref{fig:typical}.
%The left figure shows the localization before local refinement,
%and the right figure shows the localization 
%after local refinement (the steepest descent subroutine from SNL-SDP).
While the solution produced by Algorithm \ref{alg_main}
may not seem very impressive, with the help of standard local
refinement techniques we can attain very high quality solution 
even with the high number of sensors and in the presence of noise.
%The bottom figure shows in comparison the optimal solution of
%\eqref{eqn:ranklss} obtained using the same local refinement 
%routine, starting 
%at a (uniformly) randomly generated initial point.
%obtained from a projected SDP solution (i.e., the projection
%of the optimal solution of the relaxation of \eqref{eqn:ranklss}:
%\begin{equation}%\label{eqn:ranklss}
%\begin{array}{cl}\displaystyle
%\min_X &\sum_{ij\in E} |X_{ii}+X_{jj}-2X_{ij} - d_{ij}|^2
%\\
%\textrm{s.t.}&Xe=0
%%\\
%%&\rank X = r 
%\\
%&X\succeq 0,
%\end{array}
%\end{equation}
%treated also with a local refinement. )
\begin{figure}[!ht]
	\centering
	\caption{\small Illustration of
		robust facial reduction with refinement applied on an 
		instance with $1000$ sensors (no anchors) on a $[-0.5,0.5]^2$ box, 
		with noise factor $0.05$ and radio range $0.1$. 
		Top left: 
		 using Algorithm \ref{alg_main} without refinement 
            (RMSD$=28.48\% R$).
		Top right:
          using Algorithm \ref{alg_main} with refinement via 
            the steepest descent method 
            (RMSD$=1.05\% R$).
       Bottom:
		 using only the steepest descent method with a 
            randomly generated initial point 
            (RMSD$=399.45\%R$).
	    %Blue: true location; red: estimated location and discrepancy.
        Line segments represent discrepancy between estimated
        and true locations.
		\label{fig:typical}
	}
%	\epsfig{file=n1000-1.eps, height=1.55in, width=1.55in}\qquad
%	\epsfig{file=n1000-1-refined.eps, height=1.5in, width=1.5in} \qquad
%	\epsfig{file=n1000-1-steepest.eps, height=1.5in, width=1.5in}	
%	\begin{figure}
%		\centering
		\begin{subfigure}[b]{0.43\textwidth}
			\includegraphics[width=\textwidth, trim= 0.5 1 0.5 0.5, clip=true]{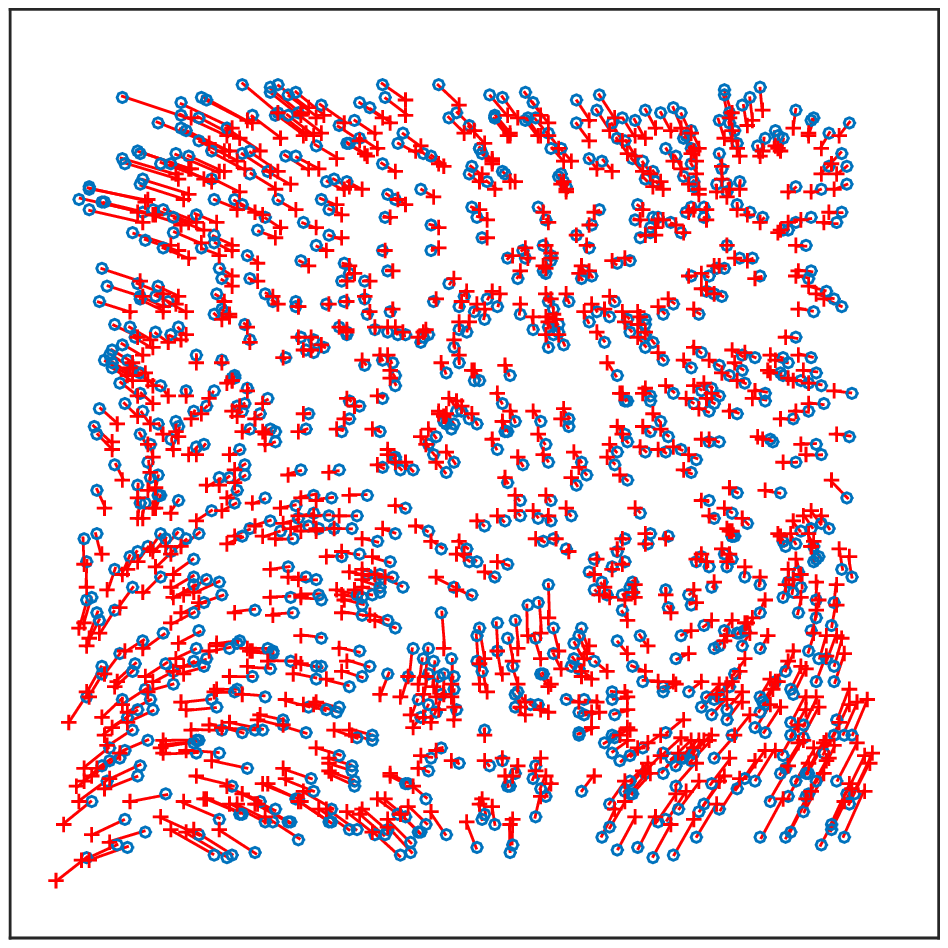}
%			\caption{A gull}
%			\label{fig:gull}
		\end{subfigure}
		\quad %add desired spacing between images, e. g. ~, \quad, \qquad, \hfill etc. 
		\begin{subfigure}[b]{0.43\textwidth}
			\includegraphics[width=\textwidth, trim= 0.5 1 0.5 0.5, clip=true]{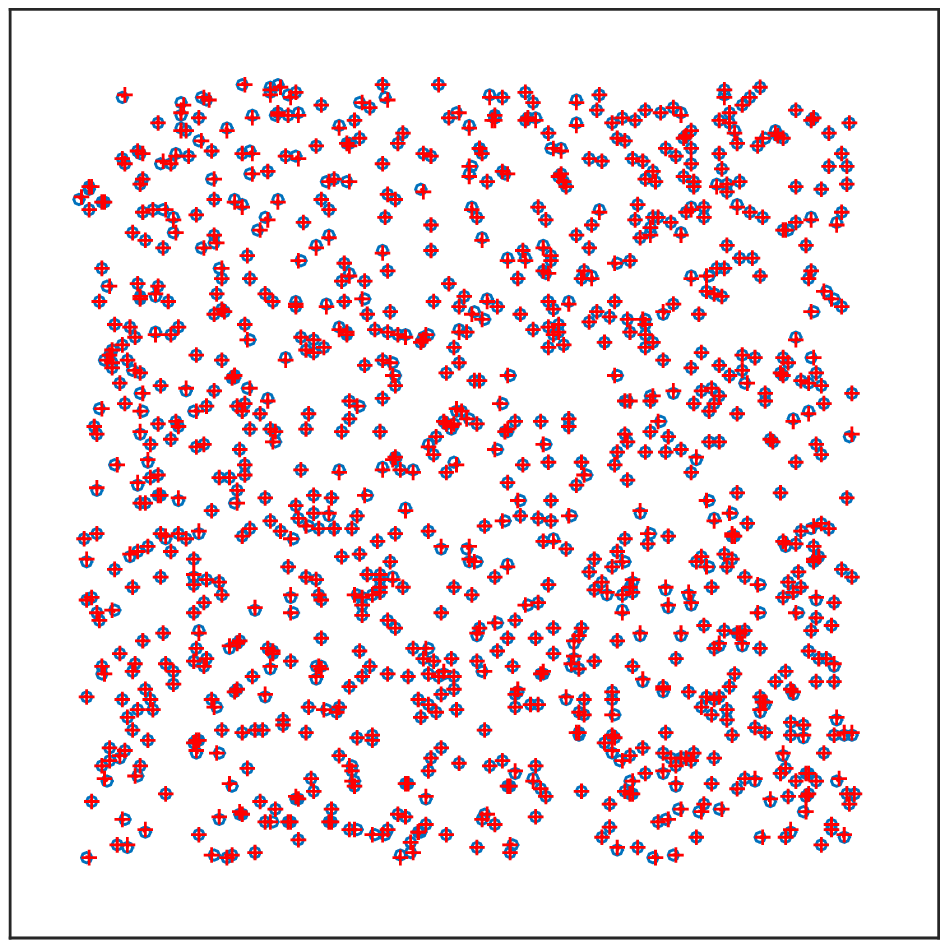}
%			\caption{A tiger}
%			\label{fig:tiger}
		\end{subfigure}
		\\\vspace{5pt} %add desired spacing between images, e. g. ~, \quad, \qquad, \hfill etc. 
		\begin{subfigure}[b]{0.43\textwidth}
			\includegraphics[width=\textwidth, trim= 0.5 0.5 0.5 1, clip=true]{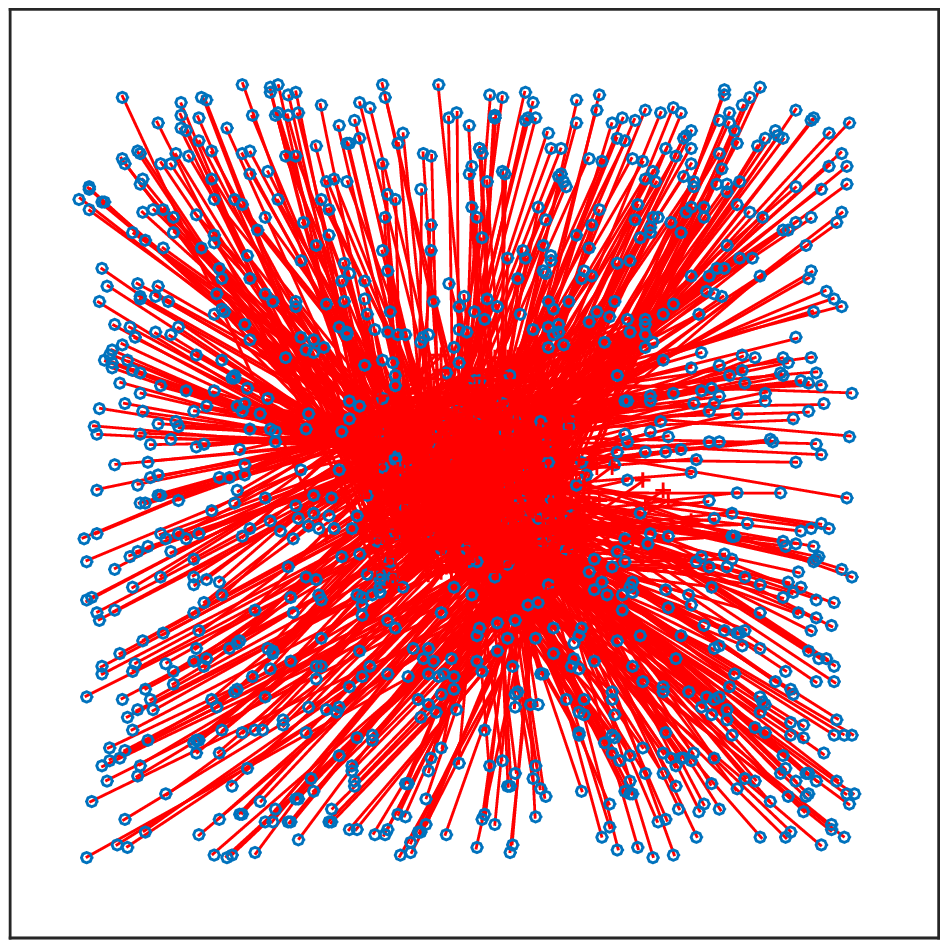}
%			\caption{A mouse}
%			\label{fig:mouse}
		\end{subfigure}
%		\caption{Pictures of animals}\label{fig:animals}
	\end{figure}

\newpage
Table \ref{table1} shows some numerical results on instances with 1000 sensors
(and no anchors) generated
as in Algorithm \ref{alg:multinoise}, with varying noise factor 
and radio range.
The tests were run on MATLAB version R2014b, 
on a Linux machine with Intel(R) Core(TM) i7-4650U CPU @ 1.70GH and
8 GB RAM.
%Intel(R) Xeon(R) CPU E5620 @ 2.40 GHz and 
%46.76 GB RAM.
We show the RMSD (as a percentage of the radio range) 
of the solutions provided by Algorithm \ref{alg_main} (in the column
``initial''),
and also the RMSD of the solution after the local refinement using the steepest descent subroutine from SNL-SDP
(in the column ``refined'').
%For comparison, we use the steepest descent refinement routine 
%from SNL-SDP on the same instances with 
%a random initial point; their RMSD values are shown in the column \texttt{Manopt}.
We see that using Algorithm \ref{alg_main} together with local refinement 
%(which by itself does not perform very well)
gives rather satisfactory results. 
The time used by Algorithm \ref{alg_main} includes the selection of cliques 
and computation of the exposing vectors, 
but excludes the postprocessing time, which is reported separately.
Table \ref{table2} shows some numerical results on larger instances.

\begin{table*}[!h]
	\centering
	\caption{Numerical results of robust facial reduction 
		on instances with 1000 vertices, 
		generated using the multiplicative noise
		model on a $[-0.5,0.5]^2$ grid. 
		Each row contains the average result over 10 instances 
		with fixed $n$, $nf$ and $R$,
		where $n$ is the number of sensors/vertices in $\GG$, 
		$R$ is the radio range;
		$nf$ is the noise factor. 
		The \emph{density} refers to the ratio 
		$\frac{\text{number of edges}}{0.5n(n-1)}$. \vspace{1em}
		\label{table1}}
	\begin{tabular}{| l | l | l | l | l | l | l | l |}
		\hline
		$n$ 
		& $nf$ 
		& $R$ 
		& density 
		&\footnotesize Time used by 
		&\footnotesize Time used for 
		&\footnotesize RMSD $\%R$ 
		&\footnotesize RMSD $\%R$ 
	\\  
		&  
		&  
		& 
		&\footnotesize Alg. \ref{alg_main} (s) 
		&\footnotesize refinement
		&\footnotesize initial 
		&\footnotesize refined 
        \\ \hline
		1000 & 0.0 & 0.25 & 15.8\% & 53.1 &  0.5 &  0.0\% &  0.0\% \\ \hline
		1000 & 0.1 & 0.25 & 15.7\% & 51.4 &  3.8 &  2.3\% &  0.6\% \\ \hline
		1000 & 0.2 & 0.25 & 15.7\% & 51.9 &  2.3 & 49.7\% &  2.0\% \\ \hline 
		1000 & 0.3 & 0.25 & 15.7\% & 67.1 &  6.5 & 76.3\% &  2.9\% \\ \hline
		1000 & 0.4 & 0.25 & 15.7\% & 64.8 &  7.0 & 72.8\% &  5.6\% \\ \hline \hline
		1000 & 0.1 & 0.15 &  6.2\% &  9.5 &  2.0 & 24.4\% &  1.1\% \\ \hline 
		1000 & 0.1 & 0.20 & 10.5\% & 20.3 &  1.5 &  4.2\% &  0.8\% \\ \hline
		1000 & 0.1 & 0.25 & 15.7\% & 51.4 &  3.8 &  2.3\% &  0.6\% \\ \hline
		1000 & 0.1 & 0.30 & 21.3\% &140.6 &  1.1 &  1.6\% &  0.5\% \\ \hline
		1000 & 0.1 & 0.35 & 27.8\% &240.6 &  1.3 &  1.2\% &  0.5\% \\ \hline 
	\end{tabular}
\end{table*}

\begin{table*}[h!]
	\centering
	\caption{Numerical results of robust facial reduction 
		on instances with more than 1000 vertices,
		generated using the multiplicative noise
		model on a $[-0.5,0.5]^2$ grid. 
		Each row contains the average result over 5 instances 
		with fixed $n$, $nf$ and $R$,
		where $n$ is the number of sensors/vertices in $\GG$, 
		$R$ is the radio range;
		$nf$ is the noise factor. 
		The \emph{density} refers to the ratio 
		$\frac{\text{number of edges}}{0.5n(n-1)}$.\vspace{1em} 
		\label{table2}}
	\begin{tabular}{| l | l | l | l | l | l | l | l |}
		\hline
		$n$ 
		& $nf$ 
		& $R$ 
		& density 
		&\footnotesize Time used 
		&\footnotesize Time used for 
		&\footnotesize RMSD $\%R$ 
		&\footnotesize RMSD $\%R$ 
	\\  
		&  
		&  
		& 
		&\footnotesize by Alg. \ref{alg_main} (s) 
		&\footnotesize refinement
		&\footnotesize initial 
		&\footnotesize refined 
        \\ \hline
		2000 & 0.1 & 0.20 & 10.6\% & 223.5 &  3.1 &  2.2\% &  0.6\% \\ \hline
		2000 & 0.2 & 0.20 & 10.5\% & 220.2 &  7.5 & 69.5\% &  2.0\% \\ \hline
		2000 & 0.3 & 0.20 & 10.5\% & 222.3 &  7.0 & 81.1\% &  3.1\% \\ \hline 
		2000 & 0.4 & 0.20 & 10.6\% & 230.2 &  6.8 & 85.1\% &  5.3\% \\ \hline \hline
		3000 & 0.1 & 0.20 & 10.5\% &1011.6 & 11.7 &  2.4\% &  0.5\% \\ \hline 
		3000 & 0.2 & 0.20 & 10.4\% & 986.5 & 23.0 & 64.3\% &  1.3\% \\ \hline 
		3000 & 0.3 & 0.20 & 10.5\% &1063.9 & 17.8 & 67.5\% &  3.0\% \\ \hline
		3000 & 0.4 & 0.20 & 10.6\% &1016.4 & 18.9 & 74.8\% &  5.0\% \\ \hline \hline
		4000 & 0.1 & 0.20 & 10.5\% &3184.0 & 13.7 &  1.8\% &  0.4\% \\ \hline 
		4000 & 0.2 & 0.20 & 10.5\% &3129.9 & 22.5 & 62.8\% &  1.3\% \\ \hline 
		4000 & 0.3 & 0.20 & 10.5\% &3226.1 & 27.3 & 79.8\% &  2.8\% \\ \hline 
		4000 & 0.4 & 0.20 & 10.6\% &3220.1 & 24.1 & 71.1\% &  4.9\% \\ \hline \hline
		4000 & 0.2 & 0.175&  8.3\% &1618.1 & 30.9 & 56.7\% &  1.5\% \\ \hline 
		4000 & 0.3 & 0.175&  8.3\% &1554.1 & 43.2 & 88.4\% &  3.2\% \\ \hline 
		4000 & 0.4 & 0.175&  8.2\% &1535.8 & 30.5 & 86.1\% &  5.7\% \\ \hline \hline
		4000 & 0.2 & 0.15 &  6.2\% & 801.9 & 41.3 & 90.5\% &  1.7\% \\ \hline 
		4000 & 0.3 & 0.15 &  6.2\% & 783.0 & 36.2 &106.4\% &  4.0\% \\ \hline 
		4000 & 0.4 & 0.15 &  6.2\% & 759.0 & 30.7 &109.1\% &  6.8\% \\ \hline \hline
		4000 & 0.2 & 0.125&  4.4\% & 616.8 & 28.1 &110.3\% &  2.1\% \\ \hline 
		4000 & 0.3 & 0.125&  4.4\% & 541.2 & 29.8 &128.3\% &  4.5\% \\ \hline 
		4000 & 0.4 & 0.125&  4.4\% & 420.7 & 31.2 &128.6\% & 13.1\% \\ \hline \hline
		5000 & 0.2 & 0.125&  4.4\% & 905.0 & 59.5 &110.8\% &  2.0\% \\ \hline 
		6000 & 0.2 & 0.125&  4.4\% &1627.2 & 67.3 & 99.6\% &  2.6\% \\ \hline 
		7000 & 0.2 & 0.125&  4.4\% &2237.8 & 93.9 &100.3\% &  1.9\% \\ \hline
		8000 & 0.2 & 0.125&  4.4\% &3704.7 &120.4 & 92.6\% &  1.9\% \\ \hline 
		9000 & 0.2 & 0.125&  4.4\% &5883.7 & 87.6 & 97.9\% &  1.9\% \\ \hline
	\end{tabular}
\end{table*}

\newpage

\section{The Pareto frontier of the unfolding heuristic.}
\label{sect:Pareto}
The facial reduction algorithm presented in the previous section is effective when $G$ is fairly dense (so that many cliques are available) and the SDP relaxation of the EDM completion problem without noise is exact. In this section, we consider problems at the opposite end of the spectrum. We will suppose that $G$ is sparse and we will seek a low rank solution approximately solving  the SDP \eqref{eqn:SDP_conv}. To this end,  consider the problem:
\begin{align}
\text{maximize}~~~ &\tr X\notag\\
\text{subject to}~~ &\|\mathcal{P}\circ\KK(X)-d\|\leq \sigma \label{prob:main}\\
~ &Xe=0\notag\\
&X\succeq 0.\notag
\end{align}
%where $\mathcal{P}$ is the canonical projection $\mathcal{P}\colon\mathcal{S}^n\to\R^{E}$.
Here, an estimate of the tolerance $\sigma>0$ on the misfit is typically available based on the physical source of the noise. %We will comment on this further at the end of the section. 
Trace maximization encourages the solution $X$ to have a 
lower rank. 
This is in contrast to the usual min-trace strategy 
in compressed sensing; 
see \cite{var_prop,par,lsqr_con} for a discussion. Indeed, as was mentioned in the introduction in terms of the factorization $X=PP^T$, the equality $\tr(X)= \frac{1}{2n}\sum^n_{i,j=1} \|p_i-p_j\|^2$ holds, where $p_i$ are the rows of $P$. Thus trace maximization serves to ``flatten'' the realization of the graph. 
We focus on the max-trace regularizer, though an entirely analogous analysis holds for min-trace. At the end of the section we compare the two.

%\subsection{Robust formulation of the noisy SNL}%
%Another common approach to the noisy SNL problem 
%is to consider the optimization problem in the robust form:
%\begin{align}
%\max_X~ &\tr X\notag\\
%\textrm{s.t.}~~ &\frac12\|\mathcal{P}\circ\KK(X)-d\|^2_2\leq \sigma \label{prob:main}\\
%~ &Xe=0\notag\\
%&X\succeq 0,\ X\in\Ss^n,\notag
%\end{align}
%where $\mathcal{P}$ is the canonical projection 
%$\mathcal{P}\colon\mathcal{S}^n\to\R^{\EE}$.
%Here, the tolerance $\sigma>0$ is typically known and 
%the Frobenius norm can be replaced by another measure of 
%deviation depending on context. 
%For simplicity we restrict to the Frobenius norm in the 
%discussion. 
%Trace maximization encourages the solution $X$ to have a 
%lower rank. 
%This is in contrast to the usual min-trace strategy 
%in compressed sensing; 
%see \cite{var_prop,par,lsqr_con} for a discussion. 
%We focus on the max-trace regularizer for concreteness, though an entirely analogous analysis holds for the min-trace. 

We propose a first-order method for this problem using a 
Pareto search strategy originating in portfolio optimization. 
This technique has recently garnered much attention in wider 
generality; e.g., \cite{par,lsqr_con,spgl1:2007}.  
The idea is simple: exchange the objective and the difficult 
constraint, and then use the easier flipped problem to solve 
the original. Thus we are led to consider the parametric optimization problem
\begin{align}
\varphi(\tau):=~~~~
\text{minimize}~~~~ &\|\mathcal{P}\circ\KK(X)-d\|\notag\\
\text{subject to}~~~~ &\tr X =\tau \label{eq:primal}\\
&Xe=0\notag\\
&X\succeq 0\notag.
\end{align}
See Figure \ref{fig:pareto_noise_compare} below for an illustration.

\begin{figure}[!ht] 
\centering 			
\begin{subfigure}[b]{0.48\textwidth} 
\includegraphics[width=\textwidth,trim= 35 1 30 1, clip=true]{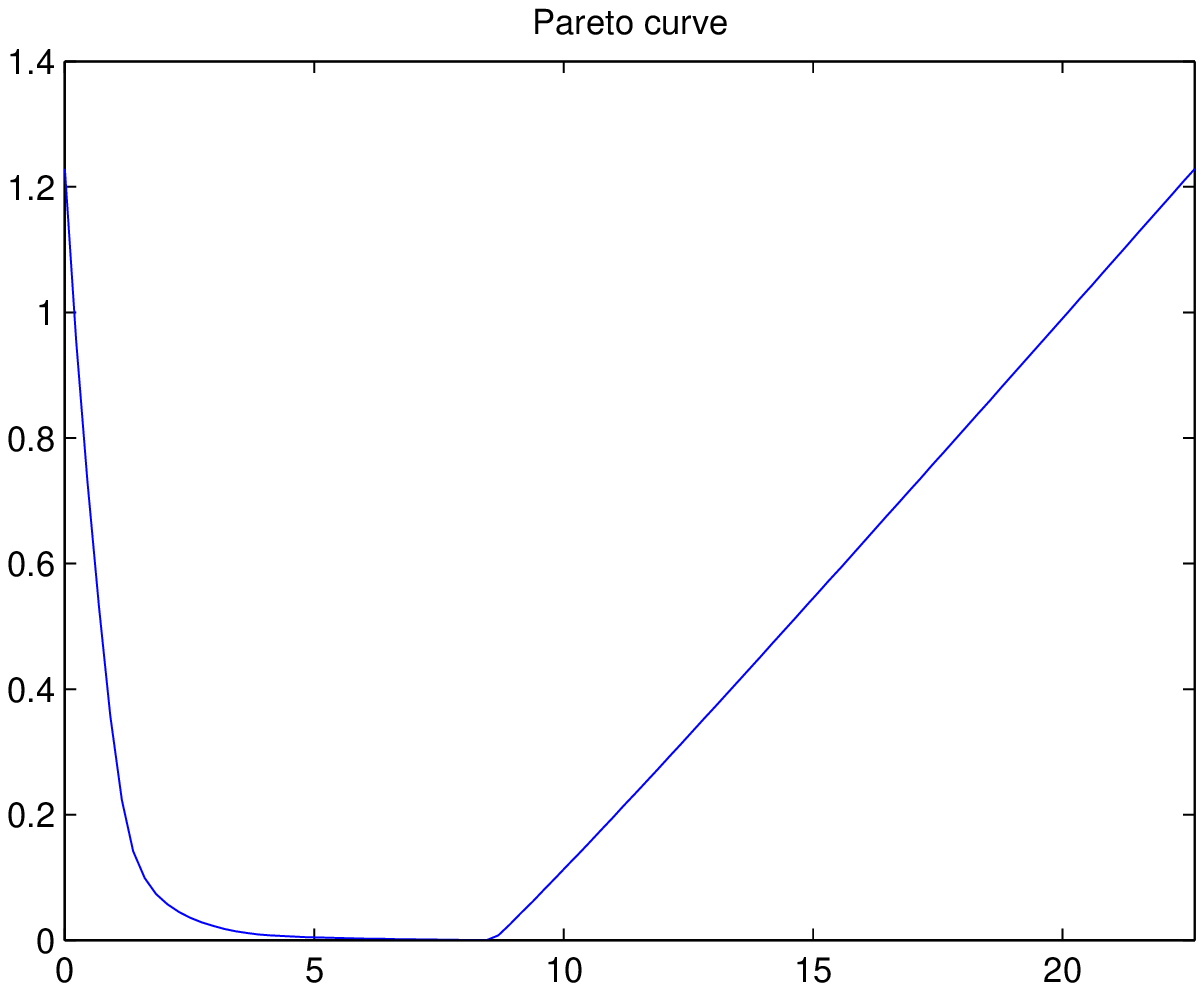} 
\caption{$\sigma=0$} 
\label{fig:pareto_nonoise} 
\end{subfigure} ~ 
\begin{subfigure}[b]{0.48\textwidth} 
\includegraphics[width=\textwidth,trim= 30 1 35 1, clip=true]{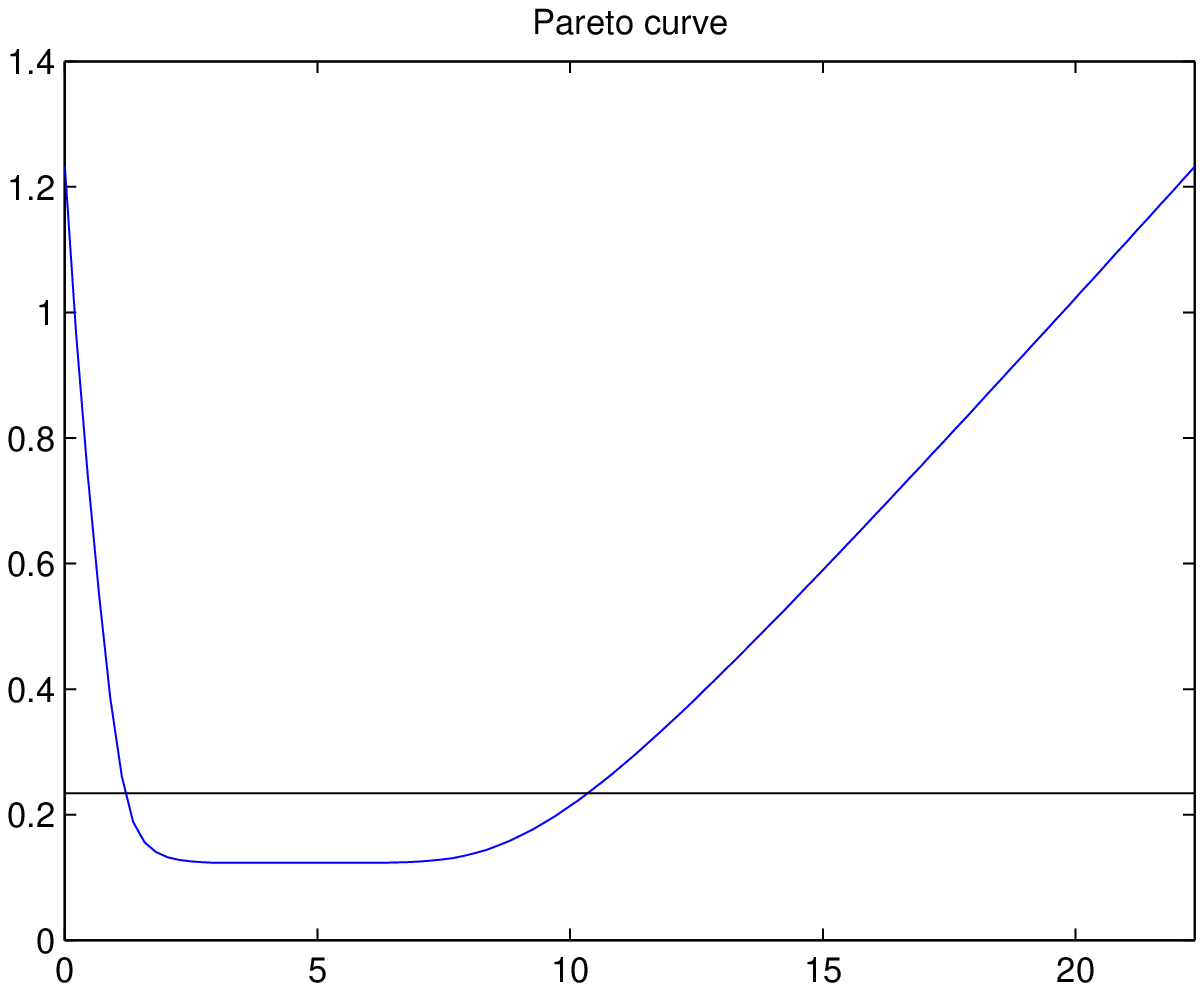} 
\caption{$\sigma=0.24$} 
\label{fig:pareto_noise} 
\end{subfigure}
\caption{Graph of $\varphi$ with noise tolerance $\sigma=0$ and $\sigma=0.24$}
\label{fig:pareto_noise_compare}
\end{figure}

%This flipped problem is easier to solve when $f_1$ or $f_2$ are in use, as we will observe shortly. %Namely, by the Eckart-Young theorem one only needs to compute an eigenvalue decomposition. As a result, we can apply a standard first order method, like projected gradient to solve this flipped problem. 
%Define now the value function $\varphi(\tau)$ to be 
%the optimal value of the flipped problem 
%with the tolerance $\tau$. 
Observe that the evaluation of $\varphi(\tau)$ is well adapted 
to first-order methods, since the feasible region is so simple.
It is well-known that $\varphi$ is a convex function, and therefore
to solve the original problem \eqref{prob:main}, 
we simply need to find the largest $\tau$ satisfying 
$\varphi(\tau)\leq \sigma$. We note that the smallest value of $\tau$ satisfying
$\varphi(\tau)\leq \sigma$ corresponds instead to minimizing the trace. We propose to evaluate $\varphi(\tau)$ by the Frank-Wolfe algorithm and then solve for the needed value of $\tau$ by an inexact Newton method. 
We will see that this leads to an {\em infeasible method} that is unaffected by the inherent ill-conditioning of the underlying EDM completion problem discussed in the previous sections.

\subsection{An inexact Newton method. }
We now describe an inexact Newton method for finding the largest value $\tau$ satisfying $\varphi(\tau) \leq \sigma $. To this end, we introduce the following definition.
\begin{definition}[Affine minorant oracle]
	{\rm
		Given a function $v : I \to \R$ on an interval $I\subset \R$, an \emph{affine minorant oracle} is a mapping $\mathcal{O}_v$ that assigns to each pair $(t, \alpha) \in I\times [1,\infty)$ real numbers $(l,u,s)$ such that $0 \leq l \leq v(t)\leq u$, $\frac{u}{l} \leq \alpha$, and the affine function $t' \mapsto l + s (t' - x)$ minorizes $v$.
	}
\end{definition}

For the EDM completion problem, the function $v$ is given by $v(\tau) = \varphi(\tau) -\sigma$. The inexact Newton method based on an affine minorant oracle is described in Algorithm~\ref{IN}.

\begin{algorithm}[h]
	\caption{Inexact Newton method\label{IN}} 
	\begin{algorithmic}
		\STATE {\bf Input:} Convex function $v\colon I\to\R$ on an interval $I\subset\R$ via an affine minorant oracle $\mathcal{O}_v$, target accuracy $\beta>0$, initial point $t_0\in I$ with $v(t_0) > 0$, and a constant $\alpha\in (1,2)$.
		\smallskip
		
		\STATE  $(l_0,u_0,s_0):=\mathcal{O}_v(t_0,\alpha)$;
		\STATE $k\gets 0$;
		\STATE $l_0\gets  0$;	
		\STATE $u_0\gets  +\infty$;			
		\WHILE{$\frac{u_{k}}{l_{k}} > \alpha$ and $u_k>\beta$}{
			\STATE $t_{k+1} \gets  t_k - \frac{ l_k }{s_k}$; 
			\STATE $(l_{k+1},s_{k+1}):=\mathcal{O}_v(t_{k+1},\alpha)$;
			\STATE $k\gets k+1$;
		} 
		\ENDWHILE 
		\STATE {\bf return} $t_k$;
	\end{algorithmic}
\end{algorithm}

%\begin{algorithm}[H]\label{alg_inexact_newt}%
%	Convex decreasing function $v\colon\R\to\R$ via an affine minorant oracle $\mathcal{O}_v$, target accuracy $\beta>0$, initial point $t_0\in\R$ with $v(t_0) > 0$, and a constant $\alpha\in (1,2)$.}
%	Set $(l_0,s_0)=\mathcal{O}_v(t_0,\alpha)$\;
%	$k=0$\;
%	$l_{-1} = \infty$\;
%	\WHILE{$l_{k} > \frac{\beta}{\alpha}$}{
%		\If{$l_k>l_{k-1}$}{
%			$l_k\gets l_{k-1}$\;
%		}
%		$t_{k+1} := t_k - \frac{ l_k }{s_k}$\;
%		$(l_{k+1},s_{k+1}):=\mathcal{O}_v(t_{k+1},\alpha)$\;
%		$k\gets k+1$;\\
%	}
%	\ENDWHILE
%	\Return $t_{k}$\;
%	\caption{Inexact Newton Method}
%\end{algorithm}
%\bigskip

It can be shown that the iterates $t_k$ generated by the inexact 
Newton method (Algorithm \ref{IN}), 
when applied to a convex function $v\colon I\to\R$ 
having a root on the interval $I$, converge to the root 
$\bar{t}$ of $v$ closest to $t_0$. 
Moreover, the convergence is linear in function value: 
the algorithm is guaranteed to terminate after at most
$$K\leq \max \left\{\log_{2/\alpha}\Big(\frac{|s_0|R}{\beta}\Big)+\log_{2/\alpha}(2)\cdot \log_{2/\alpha}\Big(\frac{2l_0}{\beta}\Big) ,1\right\}$$
iterations, where we set $R=\bar{t}-t_0$. For a proof and a discussion, 
see the preprint \cite{flippy}. 

Thus to implement this method, for the problem \eqref{prob:main}, we must describe an affine minorant oracle for $v(t)=\varphi(t)-\sigma$.
Then, after the number of iterations given above, we can obtain a centered PSD matrix $X$ satisfying 
$$\|\mathcal{P}\circ\KK(X)-d\|\leq \sigma +\beta\qquad \textrm{ and } \qquad \tr(X)\geq OPT,$$
where OPT denotes the optimal value of  \eqref{prob:main}.
A key observation is that the derivative of $v$ at the root 
{\em does not} appear in the iteration bound. 
This is important because for the function 
$v(t)=\varphi(t)-\sigma$, the inherent ill-conditioning of 
\eqref{prob:main} can lead to the derivative of $v$ at 
the root being close to zero.

\subsection{Solving the inner subproblem with Frank-Wolfe algorithm. }
In this subsection, we describe an affine minorant oracle for $\varphi(\tau)$ based on
the {\em Frank-Wolfe algorithm} \cite{FW-alg}, 
which has recently found many applications in 
machine learning (see, e.g., \cite{nem_cond,Jaggi}). Throughout, we fix a value $\tau$ satisfying $\varphi(\tau)> \sigma$. %To this end, let $V\in\R^{n\times (n-1)}$ be any matrix satisfying
%\begin{equation}
%\label{eqn:V}
%V \text{ has full rank}\qquad \textrm{and}\qquad \range(V)=e^\perp.
%\end{equation}
To apply the {\em Frank-Wolfe algorithm}, we must first square the objective in \eqref{prob:main} to make it smooth.
To simplify notation, define $$\mathcal{A}:=\PP\circ\KK,\quad f(X):=\frac{1}{2}\|\mathcal{A}(X)-d\|^2 \quad\textrm{ and } \quad \mathcal{D}:=\{X\succeq 0: \tr X=1, \,Xe=0\}.$$
Thus we seek a solution to 
%\subsubsection{Projected gradient}
%An immediate approach for solving the flipped 
%problem \eqref{eq:flipped} is to apply the standard 
%{\em gradient projection method} to the problem 
%\eqref{eq:flipped}. 
%Indeed each projection onto the feasible region 
%$\{Z\succeq 0: \tr Z\geq \tau\}$ 
%requires only a single eigenvalue decomposition. 
%Hence for moderately sized problems each iteration of 
%the method is relatively cheap. 
%In turn, recovering the Lagrange multiplier is an easy operation. 
\begin{equation*}
\min\, \{f(X): X\in \tau\mathcal{D}\}. 
\end{equation*}
The Frank-Wolfe scheme is described in Algorithm~\ref{FW}.

%Assuming that the trace constraint in \eqref{eq:flipped} 
%s tight, 
%our main problem \eqref{eq:flipped} certainly 
%fits into this framework with 
%$f(\cdot)=\frac12\|\cdot-d\|^2_2$, 
%$\mathcal{A}=\mathcal{P}\circ\widehat{\mathcal{K}}$, and $D=\{Z\succeq 0: %\tr Z=\tau, Ze=0\}$.
%The standard Frank-Wolfe algorithm then takes the following form.
\begin{algorithm}[h]
	\caption{Affine minorant oracle based on the Frank-Wolfe algorithm\label{FW}} 
	\begin{algorithmic}
		\STATE {\bf Input:} $\tau\geq 0$, relative tolerance  $\alpha > 1$, and $\beta >0$.
		\STATE Let $k\leftarrow 0$, $l_0\leftarrow \frac12\sigma^2$, and $u_0\leftarrow  +\infty$. 
		Pick any point $X_0$ in $\mathcal{\tau D}$. 
		
		\WHILE{$\sqrt{2u_k}-\sigma>\alpha (\sqrt{2l_k}-\sigma)$ and $\sqrt{2u_k}-\sigma> \beta$}{
			\STATE Choose a direction 
			\begin{equation}
			S_k\in\argmin_{S\in \tau\mathcal{D}} 
			\langle \nabla f( X_k),S\rangle\label{eq:direc};
			\end{equation}
			\STATE Set the stepsize: $\gamma_k \in \argmin_{\gamma\in [0,1]}f(X_k+\gamma (S_k-X_k))$;
			\STATE Update the iterate: $X_{k+1}\leftarrow X_k +\gamma_k (S_k-X_k)$;
			\STATE Update the upper bound: $u_{k+1}\leftarrow  f(X_{k+1})$;
			\STATE Update the lower bound: 
			\[
			l_{k+1}\leftarrow \max\left\{l_{k},f(X_k)+\langle \nabla f(X_k), S_k-X_k\rangle \right\};
			\]
			\STATE \vspace{-0.8em} Increment the iterate:
			$k\leftarrow k+1$;
		} 
		\IF{$l_{k+1}>l_k$}{
			\STATE $y\leftarrow d - \PP\circ\KK(X_k)$;\\
			$X\leftarrow X_k$\\
			$S\leftarrow S_k$
		}
		\ENDIF
		
		\ENDWHILE
		\STATE $l\leftarrow \frac{l_k+\frac{1}{2}\|y\|^2_2}{\|y\|_2}-\sigma$;\\
		$u\leftarrow \sqrt{2u_k}-\sigma$;\\
		$s=\frac{1}{\tau\|y\|}\langle \nabla f(X),S\rangle$;
		
		\STATE {\bf return} $\left(l,u,s\right)$;
	\end{algorithmic}
\end{algorithm}

%The Frank-Wolfe algorithm is guaranteed to complete in 
%$\mathcal{O}(\frac{1}{\varepsilon})$ iterations. 

%It is easy to see from the properties of the classical Frank-Wolfe algorithm that upon termination in terms of the function $v(t):=\varphi(t)-\sigma$,
%we have $$l_k-\sigma \leq \varphi(\tau)\leq u_k-\sigma.$$

The computational burden of the method is the 
minimization problem \eqref{eq:direc}. To elaborate on this, observe first that
\begin{align*}
\nabla f(X)
=\ & \mathcal{K}^*\circ\mathcal{P}^*\big(\mathcal{P}\circ{\KK}(X)-d\big).
\end{align*}
Notice that the matrix $\mathcal{K}^*\circ\mathcal{P}^*
\big(\mathcal{P}\circ{\KK}(X)-d\big)$ has the same sparsity pattern, modulo the diagonal, 
as the adjacency matrix of the graph. 
%It is important to notice that 
%in light of the standard formula 
%$\KK^*(X)=2(\Diag (Xe)-X)$ and the fact that $\mathcal{P}^*\circ\mathcal{P}$
%corresponds to the adjacency matrix of the graph, 
%the matrix $\mathcal{K}^*\circ\mathcal{P}^*
%\big(\mathcal{P}\circ{\KK}(Z_k)-d\big)$ 
%has the same sparsity pattern, modulo the diagonal, 
%as the adjacency matrix of the graph. 
As a result, when the graph $G$ is {\em sparse}, 
we claim that the linear optimization problem 
\eqref{eq:direc} is easy to solve. 
Indeed, observe $\nabla f(X)e=0$ and consequently an easy computation shows that
$\min_{S\in \tau\mathcal{D}} \,
\langle \nabla f( X),S\rangle$ equals $\tau$ times the minimal eigenvalue of the restriction of $\nabla f(X)$ to $e^{\perp}$; this minimum in turn is attained at the matrix $\tau vv^T$ where $v$ is the corresponding unit-length eigenvector. Thus to 
solve \eqref{eq:direc} we must find only the minimal eigenvalue-eigenvector pair of $\nabla f(X)$ on $e^{\perp}$, which can be done
fairly quickly by a Lanczos method, and in particular, 
by orders of magnitude faster than the full eigenvalue 
decomposition. 
Thus, the Frank-Wolfe method is perfectly 
adapted to our problem instance.
\smallskip

\begin{thm}[Affine minorant oracle]\label{thm:aff_or}
 Algorithm \ref{FW} is an affine minorant oracle for the function $v(\tau):=\varphi(\tau)-\sigma$.
\end{thm}
\begin{proof}
We first claim that upon termination of Algorithm \ref{FW}, the line 
$t'\mapsto l+s(\tau-\tau')$ is a lower minorant of $v(\tau')-\sigma$. To see this,
observe that the dual of the problem 
$$\varphi(\tau)=\min_{X\in\tau \mathcal{D}}~ \|\mathcal{A}(X)-d\|$$
is given by 
$$\max_{z\in\R^E:\, \|z\|\leq 1}~ h_{\tau}(z):=\langle d,z\rangle -\tau \delta^*_{\mathcal{D}}(\mathcal{A}^*z),$$
where $\delta^*_{\mathcal{D}}$ denotes the support function of $\mathcal{D}$.
Then by weak duality for any vector $z$ with $\|z\|_2\leq 1$ and any $\tau'$, we have the inequality
\begin{equation}\label{eqn:minorant}
\varphi(\tau')\geq h_{\tau'}(z)=\langle d,z\rangle -\tau' \delta^*_{\mathcal{D}}(\mathcal{A}^*z)=h_{\tau}(z)-(\tau'-\tau)\delta^*_{\mathcal{D}}(\mathcal{A}^*z).
\end{equation}
Hence the affine function $\tau'\mapsto h_{\tau}(z)-(\tau'-\tau)\delta^*_{\mathcal{D}}(\mathcal{A}^*z)$ minorizes the value function $\varphi(\tau')$.
%Now let $X_k$ be an iterate generated by the Frank-Wolfe method yielding a lower bound $\frac{1}{2}l^2_{k+1}=f(X_k)+\langle \nabla f(X_k), S_k-X_k\rangle$. Then in terms of the residual $r_k:=d-\mathcal{A}(X_k)$ an easy computation shows
Now a quick computation shows that upon termination of Algorithm~\ref{FW}, we have
\begin{equation}\label{eqn:id}
l_k+\frac{1}{2}\|y\|^2=h_{\tau}(y). %\langle d,r_k\rangle-\tau \delta^*_{\mathcal{D}}(\mathcal{A}^*r_k).
\end{equation}
Setting $z=\frac{y}{\|y\|_2}$ in inequality \eqref{eqn:minorant} and using the identity \eqref{eqn:id}, we obtain for all $\tau'\in\R$ the inequality 
%where $\delta^*_{\mathcal{D}}$ is the support function of $\mathcal{D}$.
%Consider now at any iteration $k$ of the Frank-Wolfe algorithm, the residual $y_k=d-\mathcal{P}\circ\mathcal{K}(X_k)$.
%Then by weak duality, we have
%\begin{align*}
%\varphi(\tau')&\geq \left\langle d,\frac{y_k}{\|y_k\|}\right\rangle -\tau' \delta^*_{\mathcal{D}}\left(\mathcal{K}^*\circ\mathcal{P}^*\left(\frac{y_k}{\|y_k\|}\right)\right)\\
%&=\frac{1}{\|y_k\|_2}\left(\langle d,y_k\rangle-\frac{1}{2}\|y_k\|^2_2-\tau \delta^*_{\mathcal{D}}(\mathcal{K}^*\circ\mathcal{P}^*(y_k)) +(\tau-\tau')\delta^*_{\mathcal{D}}(\mathcal{K}^*\circ\mathcal{P}^*(y_k))  +\frac{1}{2}\|y_k\|^2_2\right).
%\end{align*}
%The sum of the first three terms on the right-hand-side exactly equals $l_k$ and hence we deduce
\begin{align*} 
\varphi(\tau')&\geq \frac{l_k+\frac{1}{2}\|y\|^2}{\|y\|}-(\tau'-\tau)\frac{\delta^*_{\mathcal{D}}(\mathcal{A}^*y)}{\|y\|}\\
&=l+\sigma+s(\tau'-\tau).
\end{align*}
Hence the line 
$t'\mapsto l+s(\tau-\tau')$ is a lower minorant of $v(\tau')-\sigma$, as claimed.
Next, we show that upon termination, the inequality $\frac{u}{l}\leq \alpha$ holds. To see this, observe that
\[
\begin{aligned}
\frac{u}{l} &= \frac{2\|y\|u}{2l_k+\|y\|^2-2\sigma\|y\|}
\leq\frac{2\|y\| u}{\left(\frac{u + \alpha\sigma}{\alpha}\right)^2 + \|y\|^2 - 2\sigma\|y\|} \\
&=
\alpha\left(
\frac{2\|\alpha y\| u}{\left(u + \alpha\sigma\right)^2 + \|\alpha y\|^2 - 2\alpha\sigma\|\alpha y\|}
\right).
\end{aligned}
\] 
Now, observe that the numerator of the rightmost expression is always less than the denominator: 
\[
\begin{aligned}
\Big(\left(u + \alpha\sigma\right)^2 + \|\alpha y\|^2 - 2\alpha\sigma\|\alpha y\|\Big) - 2\|\alpha y\| u  = \left(u + \alpha\sigma\right)^2 + &\|\alpha y\|^2 - 2\|\alpha y\|(u + \alpha\sigma) \\
& = \left(u + \alpha\sigma - \|\alpha y\|\right)^2  \geq 0.
\end{aligned}
\]
We conclude that $\frac{u}{l}\leq \alpha$, as claimed. This completes the proof.
%We omit the details; they can be found in the preprint \cite{flippy}. 
\end{proof}
\smallskip

Thus Algorithm~\ref{FW} is an affine minorant oracle for $\varphi-\sigma$, and  
linear convergence guarantees of the inexact Newton method (Algorithm~\ref{IN}) apply.

Finally let us examine the iteration complexity of the Frank-Wolfe algorithm itself. Suppose that that for some iterate $k$, we have $\frac{\sqrt{2u_k}-\sigma}{\sqrt{2l_k}-\sigma}>\alpha$ and $\sqrt{2u_k}-\sigma > \beta$. Dropping the subscripts $k$ for clarity, observe that $\frac{\sqrt{2u}-\sqrt{2l}}{\beta}> \frac{(\sqrt{2u}-\sigma)-(\sqrt{2l}-\sigma)}{\sqrt{2u}-\sigma}> 1-\frac{1}{\alpha}$. Consequently in terms of the duality gap $\epsilon:=u-l$, we have
$$2\epsilon\geq (\sqrt{2u}-\sqrt{2l})^2 > \beta^2\left(1-\frac{1}{\alpha}\right)^2.$$ Hence Algorithm \ref{FW} terminates provided $\epsilon\leq \frac{1}{2}\beta^2\left(1-\frac{1}{\alpha}\right)^2$. Standard convergence guarantees of the Frank-Wolfe method (e.g., \cite{FW-alg,DR,Jaggi}), therefore imply that the method terminates after $\mathcal{O}\left({\frac{\tau_k L^2}{\beta^2}}\right)$ iterations, where $L$ is the Lipschitz constant of the gradient $\nabla f$.

Summarizing, consider an instance of the problem \eqref{prob:main} with optimal value $OPT$. Then given a target accuracy $\beta>0$ on the misfit $\|\PP\circ\KK(\cdot)-d\|$, we can find a matrix $X\succeq 0$ with $Xe=0$ that is super-optimal and nearly feasible, meaning
$$\tr(X)\geq \textrm{OPT}\qquad\textrm{ and }\qquad\|\PP\circ\KK(X)-d\|\leq \sigma+\beta$$
using at most $\max \left\{\log_{2/\alpha}\Big(\frac{|s_0|R}{\beta}\Big)+\log_{2/\alpha}(2)\cdot \log_{2/\alpha}\Big(\frac{2l_0}{\beta}\Big) ,1\right\}$ inexact Newton iterations\footnote{As before $|s_0|$ is the slope of the value function $v$ at $\tau_0$ and $R=\tau_0-\textrm{OPT}$.}, with each inner Frank-Wolfe algorithm terminating in at most $\mathcal{O}\left({\frac{\tau_0 L^2}{\beta^2}}\right)$ many iterations. 
Finally, we mention that in the implementation of the method, it is essential to 
warm start the Frank-Wolfe algorithm using iterates from previous Newton iterations.

\subsection{Comparison of minimal and maximal trace problems.}
It is interesting to compare the properties of the minimal trace solution 
\begin{align*}
\text{minimize}~~~ &\tr X\notag\\
\text{subject to}~~ &\|\mathcal{P}\circ\KK(X)-d\|\leq \sigma, \quad Xe=0,\quad X\succeq 0,\notag
\end{align*}
and the maximal trace solution
\begin{align*}
\text{maximize}~~~ &\tr X\notag\\
\text{subject to}~~ &\|\mathcal{P}\circ\KK(X)-d\|\leq \sigma, \quad Xe=0, \quad X\succeq 0\notag.
\end{align*}
In this section, we illustrate the difference using the proposed algorithm. Consider the following EDM completion problem coming from wireless sensor networks (Figure~\ref{fig:sensor}). The iterates generated by the inexact Newton method are plotted in Figure~\ref{fig:newtgo}.
\begin{figure}[!ht]
	\centering
	\caption{An instance of the sensor network localization problem on $n=50$ nodes with radio range $R=0.35$ and noise factor $nf=0.1$.\label{fig:sensor}}
	\includegraphics[scale=0.75,trim= 1 15 1 15, clip=true]{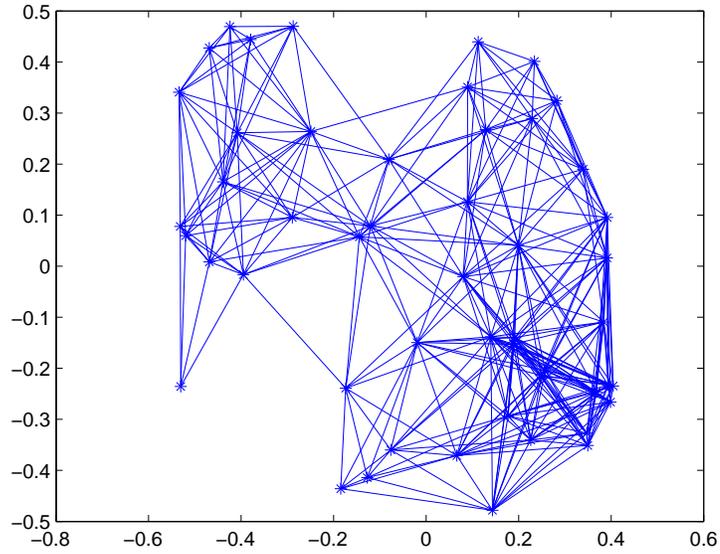}
\end{figure}

\begin{figure}[!ht]
	\centering
	\caption{Graph of $\varphi$ and inexact Newton iterates for solving the minimal trace and the maximal trace problems. Here $\sigma=0.2341$ (the dark horizontal line) and the tolerance on the misfit in the $l_2$-norm (the dashed horizontal line) is $\sigma+\beta=0.3341$.\label{fig:newtgo}}
	\includegraphics[scale=0.8,trim= 1 20 1 1, clip=true]{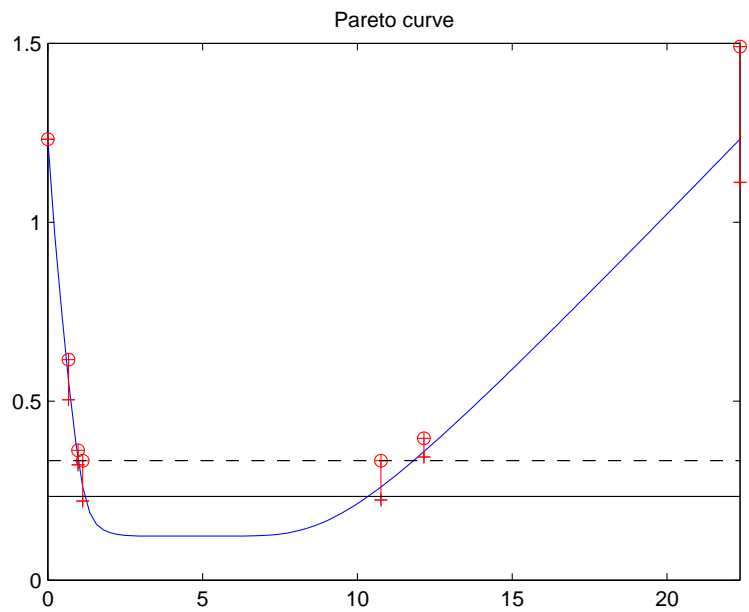}
\end{figure}

Let us consider first the maximal trace solution $X$. In Figure~\ref{fig:max_trace}, the asterisks $\textcolor{blue}{*}$ indicate the true locations of points in both pictures. In the picture on the left, the pluses $\textcolor{red}{+}$ indicate the points corresponding to the {\em maximal trace} solution $X$ after projecting $X$ onto rank $2$ PSD matrices, while in the picture on the right they denote the locations of these points after local refinement. The edges indicate the deviations.

\begin{figure}[h!]
	\centering
	\caption{Maximal trace solution.}	
	\includegraphics[scale=1.05,trim= 35 80 35 65, clip=true]{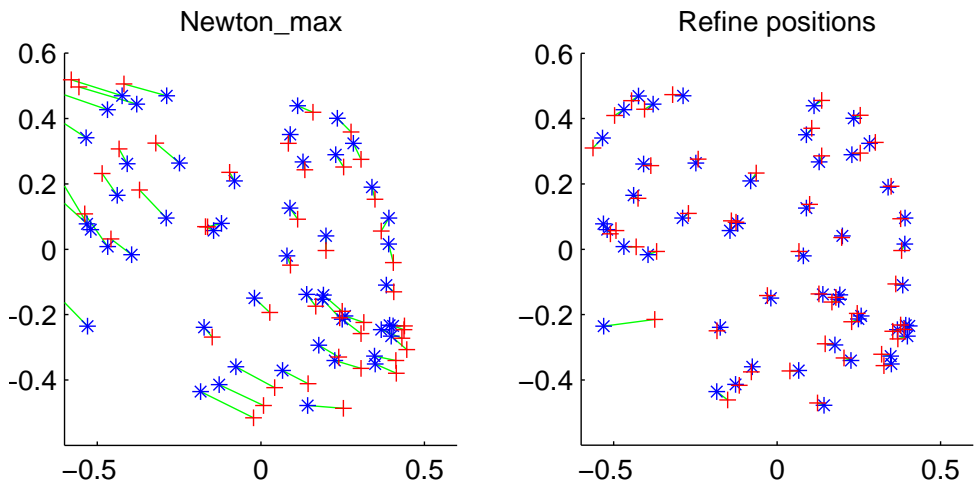}
	\label{fig:max_trace}
\end{figure}

In contrast, we now examine the minimal trace solution, Figure~\ref{fig:min_trace}. Notice that even after a local refinement stage, the realization is very far from the true realization that we seek, an indication that a local search algorithm has converged to an extraneous critical point of the least squares objective. We have found this type of behavior to be very typical in our numerical experiments.

\begin{figure}[h!]
	\centering
	\caption{Minimal trace solution.}
	\includegraphics[scale=1,trim= 35 80 35 65, clip=true]{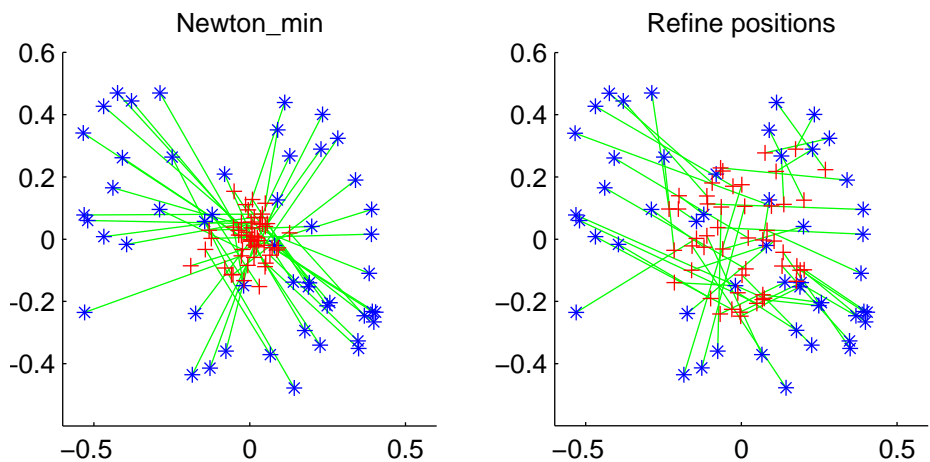}
	\label{fig:min_trace}
\end{figure}

Finally we mention an interesting difference between the maximal trace and the minimal trace solutions as far the as the value function $\varphi$ is concerned. When $\sigma=0$, the typical picture of the graph of $\varphi$ is illustrated in Figure~\ref{fig:par_fin}.
\begin{figure}[h!] 
	\centering 		 
	\includegraphics[scale=0.7, trim= 1 20 1 1, clip=true]{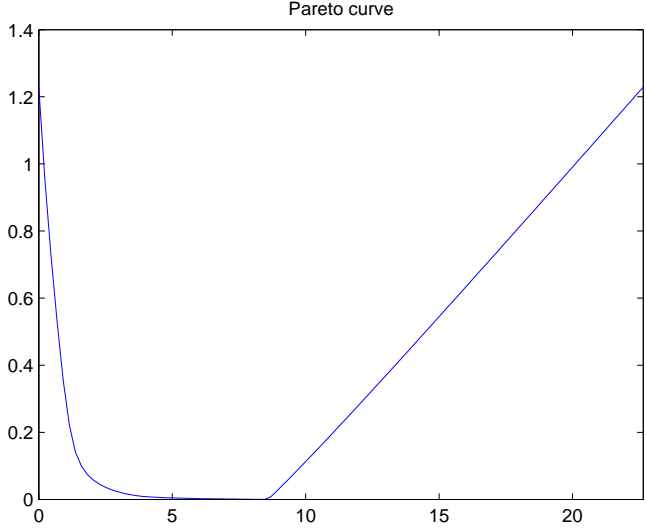} 
		\caption{Graph of $\varphi$ with $\sigma=0$.} 
		\label{fig:par_fin}
\end{figure}
The different shapes of the curve on the left and on the right sides are striking. 
To elucidate this phenomenon, consider the primal problem
\begin{align*}
\text{minimize}~~~ &\tr X\notag\\
\text{subject to}~~ &\mathcal{P}\circ\KK(X)=d, \quad Xe=0,\quad X\succeq 0,\notag
\end{align*}
and its dual
\begin{align*}
\text{maximize}~~~ &y^Td\notag\\
\text{subject to}~~ &\KK^*\circ\PP^*(y)+\beta ee^T\preceq I, \quad \notag
\end{align*}
In particular, the dual is strictly feasible and hence there is no duality gap. On the other hand, suppose that the dual optimal value is attained by some pair $(y,\beta)$ and suppose without loss of generality that $\KK^*\circ\PP^*(y)$ has an eigenvalue equal to one corresponding to an eigenvector orthogonal to $e$.
Then letting $\tau$ be the optimal value (the minimal trace), and appealing to equation~\eqref{eqn:minorant} we deduce 
\begin{align*}
\varphi(\tau')&\geq \frac{1}{\|y\|}\left(d^Ty -\tau\delta^*_{\mathcal{D}}(\mathcal\KK^*\circ\PP^*(y)) -(\tau'-\tau)\delta^*_{\mathcal{D}}(\mathcal\KK^*\circ\PP^*(y))\right)\\
 &\geq -\frac{(\tau'-\tau)}{\|y\|} \qquad \textrm{ for all }\tau'.
\end{align*}
Hence the fact that slope $\varphi'(\tau)$ is close to zero in Figure~\ref{fig:par_fin} indicates that the dual problem is either unattained (not surprising since the primal fails the Slater condition) or that the dual is attained only by vectors $y$ of very large magnitude. The reason why such phenomenon does not occur for the max-trace problem is an intriguing subject for further investigation.

\subsection{Numerical illustration.}
In this section, we illustrate the proposed method on sensor network localization instances. The data was generated in the same manner as the numerical experiments in Section~\ref{sect:numerics}. The following tables illustrate the outcome of the method by varying the noise factor ($nf$), the radio range ($R$), and the number of sensors ($n$). Throughout we have fixed the tolerance on the misfit $\|\PP\circ\KK(X)-d\|\leq \sigma+0.1$. We report the density of the graph, the CPU time that our algorithms runs, the number of the Frank-Wolfe iterations (FW\#), the RMSD of the resulting configuration, the RMSD of the configuration after local refinement, and the CPU time that the local refinement algorithm takes. The tests were run on MATLAB version R2011b, on a Linux machine with an Intel(R) Xeon(R) CPU E3-1225 @ 3.10GHz and 12 GB RAM.

\begin{table*}[h!]
\centering
\caption{Numerical results for the Pareto search strategy}
\begin{tabular}{ | r | r | r | r | r | r | r | r | r |}
\hline
\multicolumn{1}{|c|}{$n$} & 
\multicolumn{1}{c|}{$nf$} & 
\multicolumn{1}{c|}{$R$} & 
\multicolumn{1}{c|}{density} & 
\multicolumn{1}{c|}{CPU} & 
\multicolumn{1}{c|}{FW\#} & 
\multicolumn{1}{c|}{\footnotesize RMSD \%R} & 
\multicolumn{1}{c|}{\footnotesize RMSD \%R} & 
\multicolumn{1}{c|}{\footnotesize Refine} \\
\multicolumn{1}{|c|}{} & 
\multicolumn{1}{c|}{} & 
\multicolumn{1}{c|}{} & 
\multicolumn{1}{c|}{} & 
\multicolumn{1}{c|}{time (s)} & 
\multicolumn{1}{c|}{} & 
\multicolumn{1}{c|}{\footnotesize initial} & 
\multicolumn{1}{c|}{\footnotesize refined} & 
\multicolumn{1}{c|}{\footnotesize time (s)} \\\hline
 1000 &  0.0 &   0.10 &    2.9\% &      9.2 &  181 &   36.3\% &    0.1\% &    1.2 \\ \hline
 1000 &  0.1 &   0.10 &    2.9\% &      8.8 &  147 &   59.6\% &    3.6\% &    1.2 \\ \hline
 1000 &  0.2 &   0.10 &    2.9\% &      7.3 &  136 &   89.6\% &    7.5\% &    1.2 \\ \hline
 1000 &  0.3 &   0.10 &    2.9\% &      7.9 &  140 &  115.1\% &   11.8\% &    1.2 \\ \hline
\hline
 1000 &  0.1 &   0.10 &    2.9\% &      8.9 &  147 &   59.6\% &    3.6\% &    1.2 \\ \hline
 1000 &  0.1 &   0.15 &    6.3\% &      6.6 &  176 &   22.6\% &    2.1\% &    1.1 \\ \hline
 1000 &  0.1 &   0.20 &   10.7\% &     12.4 &  356 &   11.5\% &    1.4\% &    1.3 \\ \hline
 1000 &  0.1 &   0.25 &   15.9\% &     20.3 &  586 &    7.3\% &    1.2\% &    1.6 \\ \hline
 1000 &  0.1 &   0.30 &   22.0\% &     45.0 & 1074 &    4.9\% &    0.9\% &    1.4 \\ \hline
\hline
 1000 &  0.2 &   0.10 &    2.9\% &      7.3 &  136 &   89.6\% &    7.5\% &    1.2 \\ \hline
 2000 &  0.2 &   0.10 &    2.9\% &     17.1 &  169 &   66.3\% &    4.7\% &    5.0 \\ \hline
 3000 &  0.2 &   0.10 &    2.9\% &     30.8 &  189 &   56.4\% &    3.5\% &    5.0 \\ \hline
 4000 &  0.2 &   0.08 &    1.9\% &     63.8 &  227 &   80.6\% &    3.7\% &   11.6 \\ \hline
 5000 &  0.2 &   0.08 &    1.9\% &     75.1 &  179 &   74.0\% &    3.3\% &   16.9 \\ \hline
 6000 &  0.2 &   0.08 &    1.9\% &    179.6 &  264 &   68.3\% &    3.0\% &   26.9 \\ \hline
 7000 &  0.2 &   0.06 &    1.1\% &    253.7 &  345 &  119.1\% &    4.2\% &   28.8 \\ \hline
 8000 &  0.2 &   0.06 &    1.1\% &    355.4 &  370 &  112.0\% &    3.5\% &   25.8 \\ \hline
 9000 &  0.2 &   0.06 &    1.1\% &    425.8 &  338 &  108.0\% &    3.4\% &   42.4 \\ \hline
10000 &  0.2 &   0.06 &    1.1\% &    611.9 &  408 &  101.9\% &    3.1\% &   55.1 \\ \hline
11000 &  0.2 &   0.05 &    0.8\% &    744.9 &  435 &  149.3\% &    3.8\% &   39.5 \\ \hline
12000 &  0.2 &   0.05 &    0.8\% &    981.4 &  498 &  143.1\% &    3.9\% &   36.1 \\ \hline
13000 &  0.2 &   0.05 &    0.8\% &   1240.6 &  526 &  138.4\% &    4.5\% &   67.3 \\ \hline
14000 &  0.2 &   0.05 &    0.8\% &   1219.4 &  468 &  131.8\% &    6.7\% &   80.4 \\ \hline
15000 &  0.2 &   0.05 &    0.8\% &   1518.8 &  490 &  131.0\% &    5.1\% &   89.2 \\ \hline
\end{tabular}
\end{table*}

\newpage

\section{Conclusion and work in progress.}

In this paper, we described two algorithms (robust facial reduction and a search along the 
Pareto frontier) to solve the EDM completion problem with possibly inaccurate distance measurements, which has important applications and is numerically challenging.
The two algorithms are intended for EDM completion problems of different densities:
the Pareto frontier algorithm discussed in Section \ref{sect:Pareto}
is designed for sparse graphs whereas 
the robust facial reduction outlined 
in Algorithm \ref{alg_main} in Section \ref{sect:rfr}
tends to work better for denser graphs.
%As we illustrated, both approaches can work in anchorless sensor networks,
%since only the inter-sensor distance measurements are used.
%Moreover, availability of anchor information will only 
%improve the performance of the algorithms. 
%We outlined some intuitive ideas behind the robustness of the two approaches,
%and the formal statements of their robustness guarantees will be presented in our upcoming journal paper. 
%We demonstrated some numerical results 
%or the robust facial reduction technique. 
Though not studied in this work, 
it is possible to develop a distributed implementation of 
the robust facial reduction technique in order to solve even larger scale completion problems. %In principle, the robust facial reduction technique should work with networks that are sparser than those we have tested, but we have yet to verify this in practice.
The Pareto frontier estimation technique is promising for handling large scale EDM completion problems, since first-order methods become immediately applicable and sparsity of the underlying graph can be exploited when searching for a maximum eigenvalue-eigenvector pair via a Lanczos procedure. Numerical experiments have illustrated the effectiveness of both strategies.

\appendix

\section{Nearest-point mapping to $\Ss^{k,r}_{c,+}$}
\label{sect:nearest_point}

We now describe how to evaluate the nearest-point-mapping 
to the set $\Ss^{k,r}_{c,+}$---an easy and standard 
operation due to the Eckart-Young Theorem. 
To describe this operation, consider any matrix 
$X\in\mathcal{S}^n$ and a set 
$\mathcal{Q}\subset\mathcal{S}^n$. 
Define the {\em distance function} and the {\em projection}, 
respectively:
\index{distance function, $\dist(\cdot)$}
\index{$\dist(\cdot)$, distance function}
\index{projection function, $\proj(\cdot)$}
\index{$\proj(\cdot)$, projection function} 
$$\dist(X;\mathcal{Q})=\inf_{Y\in \mathcal{Q}} \|X-Y\|_F,$$
$$\proj(X;\mathcal{Q})=\{Y\in \mathcal{Q}: \|X-Y\|_F=\dist(X; \mathcal{Q})\}.$$
In this notation, we would like to find a matrix $Y$ in the set $\proj(X;\Ss^{k,r}_{c,+})$.
To this end, let 
$\begin{bmatrix}
\frac{1}{\sqrt k}e & U  \cr
\end{bmatrix}$
be any 
$k\times k$ orthogonal matrix.
First dealing with the centering constraint, one can verify $$\proj \big(X; \mathcal{S}^{k,r}_{c,+}\big)= U\Big[\proj\big(U^TXU; \mathcal{S}^{k-1,r}_+\big)\Big]U^T.$$
On the other hand, we have 
\begin{align*}
\proj \big(Z; &\mathcal{S}^{k-1,r}_+\big)=
%\\&= 
W\Diag\Big(0,\ldots,0,\lambda^+_{k-r}(Z),\ldots,\lambda^+_{k-1}(Z)\Big)W^T,
\end{align*}
where $\lambda_1(Z)\leq \ldots\leq\lambda_{k-1}(Z)$ are the eigenvalues of $Z$ and the subscript $\lambda^+_i(Z)$ refers to their positive part, and $W$
is any orthogonal matrix in the eigenvalue decomposition $Z=W\Diag(\lambda(Z))W^T$. Thus computing a matrix in $\proj(X;\Ss^{k,r}_{c,+})$ requires no more than an eigenvalue decomposition.

\section{Solving the small least squares problem}
\label{sect:lss}
We now describe how to easily solve the least squares 
system \eqref{eqn:small_sys}. 
Typically, the matrix $Y$ will have rank $n-r$. 
Then the face $\Ss^{n}_{c,+}\cap Y^{\perp}$ can be written as 
$\Ss^{n}_{c,+}\cap Y^{\perp}=U\mathcal{S}^r_+U^T$, 
where the $n\times r$ matrix $U$ has as columns an orthonormal 
basis for the kernel of $Y$.
Consequently we are interested in solving an optimization problem of the form
\begin{align*}
\min_Z~ &\|\mathcal{A}(Z)-d\|^2_2\\
\textrm{ s.t.}~~&Z \in \mathcal{S}^r_+, 
\end{align*}
where the linear operator $\mathcal{A}\colon\mathcal{S}^n\to\R^{\EE}$ 
is defined by $[\mathcal{A}(Z)]_{ij}=[\KK(UZU^T)]_{ij}$ 
for all $ij\in \EE$. 
Let $\svec(Z)$ be the vectorization of $Z$ 
and let $A$ be a $|\EE|\times \frac{r(r+1)}{2}$ 
matrix representation of the operator $\mathcal{A}$.
Thus we are interested in solving the system
\begin{align}\label{eqn:small_system}
\min_Z~ &\|A(\svec Z)-d\|^2_2\\
\textrm{ s.t.}~~&Z \in \mathcal{S}^r_+, \notag
\end{align}
where $A$ is a tall-skinny matrix. 
One approach now is simply to expand the objective 
\begin{align*}
\|A(\svec Z)-d\|^2_2=\langle (A^TA)(&\svec Z),\svec Z\rangle 
%\\ &
-2\langle A^Td,\svec Z\rangle+\|d\|^2,
\end{align*}
and then apply any standard iterative method to solve the problem 
\eqref{eqn:small_system}. 
Alternatively, one may first form an economic QR factorization 
$A=QR$ (where $Q\in\R^{|\EE|\times \frac{1}{2}r(r+1)}$
has orthonormal columns and $R\in\R^{ \frac{1}{2}r(r+1)\times
	\frac{1}{2}r(r+1)}$ is upper triangular)
and then write the objective as 
$\|A(\svec Z)-d\|^2_2=\|R(\svec Z)- Q^Td\|^2$.
We can then pose the problem \eqref{eqn:small_system} as a 
small linear optimization problem over the product of the 
semidefinite cone $\mathcal{S}^r_+$ and 
a small second-order cone of dimension 
$\R^{\frac{r(r+1)}{2}}$, 
and quickly solve it by an off-the-shelf Interior Point Method.

In practice, very often the cone constraint in 
\eqref{eqn:small_sys} is \emph{inactive}. The reason is that under reasonable conditions (see Theorem~\ref{obs:unique_sol}), in a noiseless situation, there is a unique solution to the equation $\mathcal{A}(Z)=d$, which happens to be positive definite. Hence by the robustness guarantees (Theorem~\ref{thm:robust}) a small amount of noise in $d$ will lead to a matrix solving $\min_Z \|A(\svec Z)-d\|^2_2$ that is automatically positive definite. Heuristically,
we can simply drop the cone constraint in \eqref{eqn:small_sys}
and consider the unconstrained least squares problem 
\begin{equation}
\label{eqn:small_sys_unconstr}
\min_Z~ \|A(\svec Z)-d\|^2_2,
\end{equation}
which can be solved very efficiently by classical methods.
%If the unique solution of \eqref{eqn:small_sys_unconstr}
%is positive definite (which is often the case), then 
%it is immediately the unique solution of \eqref{eqn:small_sys}.
With this observation, we often can solve \eqref{eqn:small_sys}
\emph{without using any optimization software}.

\section{Robustness of facial reduction}\label{sec:robustness} In this section, we provide rudimentary robustness guarantees on the Algorithm~\ref{alg_main}. To this end,
consider two $n\times r$ matrices $U$ and $V$, each with orthonormal columns. Then the {\em principal angles} between $\range U$ and $\range V$ are the arccosines of the singular values of $U^TV$. We will denote the vector of principal angles between these subspaces, arranged in nondecreasing order, by $\Gamma$. The symbols $\sin^k (\Gamma)$ and $\cos^k (\Gamma)$ will have obvious meanings. Thus the vector of singular values $\sigma(U^TV)$, arranged in nondecreasing order, coincides with $\cos (\Gamma)$. Consequently
in terms of the matrix $$\Delta=I-(V^TU)(V^TU)^T,$$ 
the eigenvalue vector $\lambda(\Delta)$ coincides with $\sin^2 (\Gamma)$. An important property is that the principal angles between $\range U$ and $\range V$ and the principal angles between $(\range U)^{\perp}$ and $(\range V)^{\perp}$, coincide modulo extra $\frac{\pi}{2}$ angles that appear for dimensional reasons. The following is a deep result that is fundamental to our analysis \cite{MR0246155,MR0180852, MR0264450}. It estimates the deviation in range spaces of matrices that are nearby in norm.
\smallskip

\begin{thm}[Distances and principal angles]\label{thm:dist_princ}
Consider two matrices $X,Y\in \mathcal{S}^n_+$ of rank $r$ and let $\Gamma$ be the vector of principal angles between $\range X$ and $\range Y$. Then the inequality 
$$\|\sin (\Gamma)\| \leq \frac{\|X-Y\|}{\delta(X,Y)}\qquad \textrm{ holds},$$
where $\delta(X,Y):=\min\{\lambda_r(X),\lambda_r(Y)\}$.
\end{thm}
\smallskip

The following is immediate now.
\smallskip
\begin{cor}[Deviation in exposing vectors]\label{cor:dev_expos}
Consider two rank $r$ matrices $X,Y\in\mathcal{S}^n_+$ and let $U$ and $V$ be $n\times r$ matrices with orthonormal columns that span $\ker X$ and $\ker Y$ respectively.  Then we have
$$\|UU^T-VV^T\| =\sqrt{2} \left(\frac{\|X-Y\|}{\delta(X,Y)}\right).$$
\end{cor}
\begin{proof}
Observe $\|UU^T-VV^T\|^2= 2\tr(I-(V^TU)(V^TU)^T)=2\|\sin (\Theta)\|^2$. Applying Theorem~\ref{thm:dist_princ}, the result follows.
\end{proof}
\smallskip

Next, we will need the following lemma.
\smallskip
\begin{lem}[Projections onto subsets of symmetric matrices]\label{lem:proj_sub}
	For any $n\times r$-matrix $U$ with orthonormal columns, and a matrix $X\in\mathcal{S}^n$, we have
	\begin{equation}\label{eqn:proj_lin}
	\proj(X;U\Ss^r U^T)=UU^TXUU^T,
	\end{equation}
	and for any subset $\mathcal{Q}\in \Ss^r$, we have
	\begin{equation}\label{eqn:proj_set}
	\proj(X;U\mathcal{Q} U^T)=U\proj(U^TXU; \mathcal{Q})U^T.
	\end{equation}
\end{lem}
\begin{proof}
	Optimality conditions for the optimization problem
	$$\min_{Y\in \Ss^r} \|X- U Y U^T\|^2$$
	immediately imply 
	$\proj(X;U\Ss^r U^T)=UU^TXUU^T$. Since $U\mathcal{Q} U^T$ is contained in the linear space $U\Ss^r U^T$, the projection $\proj(X;U\mathcal{Q} U^T)$ factors into a composition
	$$\proj(X;U\mathcal{Q} U^T)=\proj\Big(\proj(X;U\Ss^r U^T);  U\mathcal{Q} U^T\Big),$$
	Combining this with equation \eqref{eqn:proj_lin} we deduce
	$$\proj(X;U\mathcal{Q} U^T)=\proj\Big(UU^TXUU^T;  U\mathcal{Q} U^T\Big).$$
	On the other hand, since the columns of $U$ are orthonormal, for any $Y\in \Ss^r$ we clearly have
	$$\|UU^TXUU^T- UYU^T\|=\|U^TXU- Y\|,$$
	and equation \eqref{eqn:proj_set} follows immediately.
\end{proof}

\smallskip

\begin{cor}[Distances between faces]\label{lem:dist_subs}
Consider two $n\times r$ matrices $U$ and $V$, each with orthonormal columns and let $\Gamma$ be the vector of principal angles between $\range U$ and $\range V$.
%Then for any matrix $X=VZV^T$, we have equality
%$$\frac{1}{2}\dist^2(X; U\mathcal{S}^r U^T)=  \|\Delta^{\frac{1}{2}} Z\|^2-\frac{1}{2}\|\Delta^{\frac{1}{2}}Z\Delta^{\frac{1}{2}}\|^2.$$
Then for any $Z\in \mathcal{S}^r_+$ the estimate holds:
$$\dist(VZV^T; U\mathcal{S}^r_+ U^T)\leq \sqrt{2}\cdot\|Z\|\cdot \|\sin(\Gamma)\|.$$
%$$\dist(X; U\mathcal{S}^r U^T)\leq \|X\|\cdot \sqrt{\|\Delta\|^2+2\|\Delta\|}.$$
\end{cor}
\begin{proof}
%Consider now a matrix $X=VZV^T$ for some $Z\in V\mathcal{S}^rV^T$. 
Appealing to Lemma~\ref{lem:proj_sub}, we obtain the equation $\proj(VZV^T;U\mathcal{S}^r_+ U^T)=UU^T(VZV^T)UU^T$. Define now the matrix $\Delta=I-(V^TU)(V^TU)^T$. We successively deduce
\begin{align*}
\dist^2(&VZV^T; U\mathcal{S}^r_+ U^T)=\|VZV^T-UU^T(VZV^T)UU^T\|^2\\
&= \|VZV^T\|^2-2\tr(VZV^TUU^TVZV^TUU^T) +\tr(U^TVZV^TUU^TVZV^TU)\\
&=\|Z\|^2-2\tr\Big(\big(Z(V^TU)(V^TU)^T\big)^2\Big)+ \tr\Big(\big(Z(V^TU)(V^TU)^T\big)^2\Big)\\
&=\tr\Big(Z^2- \big(Z(V^TU)(V^TU)^T\big)^2\Big) \\
&=\tr\Big(Z^2- \big(Z-Z\Delta\big)^2\Big)=\tr\Big(2Z^2\Delta- Z\Delta Z\Delta \Big)=2\|\Delta^{\frac{1}{2}} Z\|^2-\|\Delta^{\frac{1}{2}}Z\Delta^{\frac{1}{2}}\|^2.
\end{align*}
Hence we deduce
\begin{align*}
\dist^2(VZV^T; U\mathcal{S}^r U^T)&=2\tr(Z^2\Delta)- \|Z^{\frac{1}{2}}\Delta Z^{\frac{1}{2}}\|^2\leq 2\tr(Z^2\Delta)\leq 2 \cdot\|Z\|^2\cdot\|\Lambda\|\\
&=2 \cdot\|Z\|^2\cdot\|\sin^2(\Theta)\|= 2 \cdot\|Z\|^2\cdot\|\sin(\Theta)\|^2.
\end{align*}
The result follows.
%On the other hand, 
%using von Neumann's trace inequality and Lipschitz continuity of eigenvalues, we obtain
%\begin{align*}
%\tr\Big(2Z^2\Delta- Z\Delta Z\Delta \Big)&=\tr\Big(Z\Delta Z (2I- \Delta) \Big)\leq \langle \lambda(Z\Lambda Z),\lambda(2I-\Delta)\rangle \\
%&\leq \|\lambda(Z\Lambda Z)\|\cdot\|\lambda(2I-\Delta)\|\leq 2\sqrt{n}\cdot\|X\|^2\cdot \|\Lambda\|.
%\end{align*}
%The result follows.
\end{proof}

\smallskip
%\begin{cor}[Distance between faces]\label{cor:dist_face}
%Consider two matrices $X,Y\in \mathcal{S}^n_+$, having rank $r$. Then for any matrix $\widehat{X}\in \face(X,\mathcal{S}^n_+)$, the inequality
%$$\dist\Big(\widehat{X}; \face(Y,\mathcal{S}^n_+)\Big)\leq \sqrt{2}\cdot\|\widehat{X}\|\cdot \frac{\|X-Y\|}{\delta(X,Y)} \qquad\textrm{ holds},$$
%where $\delta(X,Y):=\min\{\lambda_r(X),\lambda_r(Y)\}$.
%\end{cor}

We are now ready to formally prove robustness guarantees on the method. For simplicity, we will assume that the exposing matrices $W_{\alpha}$ are of the form $UU^T$ where $U$ have orthonormal columns, and that $\omega_{\alpha}(d)=1$ for all cliques $\alpha$ and all $d\in\R^E$. The arguments can be easily adapted to a more general setting.
For any subgraph $H$ of $G$, 
we let $d[H]$ denote the restriction of $d$ to $H$.
Following \cite{SoYe:05}, the EDM completion problem is said to be 
{\em uniquely $r$-localizable} if 
either of the equivalent conditions in 
Observation~\ref{obs:unique_sol} holds. 
\index{uniquely $r$-localizable}
In what follows, let $Alg(d)$ be the output of 
Algorithm \ref{alg_main} on the EDM completion problem.
\index{$Alg(\cdot)$, output of Algorithm \ref{alg_main}}

\smallskip
\begin{thm}[Robustness]\label{thm:robust}
%	Consider an SNL instance on the graph $\GG=(\VV,\EE)$ with 
%	noiseless distance measurements $d\in\R^{\EE}$. 
	Suppose the following: 
	\begin{itemize}
		\item for any clique $\alpha\in \Theta$, 
		the subgraph on $\alpha$ has embedding dimension $r$;
		\item the EDM completion problem is uniquely $r$-localizable 
		and Alg(d) is the realization of $G$.
		\item the matrix $Y$ obtained during the run on the noiseless problem has rank $n-r$;
	\end{itemize}
	Then there exist constants $\varepsilon >0$ and $\kappa >0$ 
	so that 
	\[
	\|\PP\circ\KK(Alg(\hat{d}))-\hat{d}\|
	\leq \kappa \|\hat{d}-d\|
	\textrm{ \small{whenever} } \|\hat{d}-d\| <\varepsilon.
	\]
\end{thm}
\begin{proof}
Throughout the proof, we will use the hat superscript to denote the objects (e.g. $\widehat{X}_{\alpha}$, $\widehat{W}_{\alpha}$) generated by Algorithm~\ref{alg_main} when it is run with the distance measurements $\widehat{d}\in\R^E$.	
Clearly for any $\hat{d}\in\R^E$, we have $\|\KK^{\dagger}\hat{d}_{\alpha}-\KK^{\dagger}d_{\alpha}\|=\mathcal{O}(\|d_{\alpha}-\hat{d}_{\alpha}\|)$ for any clique $\alpha\in \Theta$. Fix any such clique $\alpha$, and notice by our assumptions $\KK^{\dagger}d_{\alpha}$ has rank $r$.	Consequently 
$\|\widehat{X}_{\alpha}-X_{\alpha}\|=\mathcal{O}(\|d_{\alpha}-\hat{d}_{\alpha}\|)$ whenever $\hat{d}$ is sufficiently close to $d$. Appealing then to Corollary~\ref{cor:dev_expos}, we deduce $\|\widehat{W}_{\alpha}-W_{\alpha}\|=\mathcal{O}(\|\widehat{X}_{\alpha}-X_{\alpha}\|)=\mathcal{O}(\|d_{\alpha}-\hat{d}_{\alpha}\|)$. Hence $\|\widehat{W}-W\|=\mathcal{O}(\|d-\hat{d}\|)$ for all $\hat{d}$ sufficiently close to $d$. Since $W$ has rank $n-r$, we deduce $\|\widehat{Y}-Y\|=\mathcal{O}(\|d-\hat{d}\|)$.
Appealing to Theorem~\ref{thm:dist_princ}, we then deduce $\|\sin(\Gamma)\|=O(\|d-\hat{d}\|)$, where 
$\Gamma$ is the principle angle vector between the null spaces of $\widehat{Y}$ and $Y$. By Corollary~\ref{lem:dist_subs}, then 
$$\dist\Big(X; \face(\widehat{X},\mathcal{S}^n_+)\Big)= \mathcal{O}(\|\hat{d}-d\|).$$ 
The result follows.
\end{proof}

%singularity degree is one
%embdeg of each clique is one
%error in deviation

%\section{The use of appendices}
%The \verb|\appendix| command may be used before the final sections
%of a paper to designate them as appendices. Once \verb|\appendix|
%is called, all subsequent sections will appear as 
%
%\appendix
%\section{Title of appendix} Each one will be sequentially lettered
%instead of numbered. Theorem-like environments, subsections,
%and equations will also have the section number changed to a letter.
%
%If there is only {\em one} appendix, however, the \verb|\Appendix|
%(with a capital letter) should be used instead. This produces only 
%the word {\bf Appendix} in the section title, and does not add a letter.
%Equation numbers, theorem numbers and subsections of the appendix
%will have the letter ``A''  designating the section number.
%
%If you don't want to title your appendix, and just call it
%{\bf Appendix A.} for example, use \verb|\appendix\section*{}| 
%and don't include anything in the title field. This works
%opposite to the way \verb|\section*| usually works, by including the
%section number, but not using a title.
%
% 
%Appendices should appear before the bibliography section, not after,
%and any acknowledgments should be placed after the appendices and before
%the bibliography. 
%
%%\clearpage
%%\addcontentsline{toc}{section}{Index}
%%\printindex

\bigskip
{\bf Acknowledgments.}
We thank Sasha Aravkin for pointing out a part of the proof of Theorem~\ref{thm:aff_or}.

\bibliographystyle{siam}
\bibliography{master,EDM,psd}

\end{document}